\documentclass{article}

\usepackage{arxiv}

\usepackage[utf8]{inputenc} 
\usepackage[T1]{fontenc}    
\usepackage[colorlinks,citecolor=blue,urlcolor=blue]{hyperref}       
\usepackage{url}            
\usepackage{booktabs}       
\usepackage{amsfonts}       
\usepackage{nicefrac}       
\usepackage{microtype}      
\usepackage{lipsum}		
\usepackage{graphicx}
\usepackage{natbib}
\usepackage{doi}
\usepackage{amsthm}
\usepackage{amsmath,amsfonts,amssymb}  
\usepackage{dsfont}
\usepackage{algorithm}
\usepackage{algorithmic}

\newcommand{\rmd}{\mathrm{d}}
\newcommand{\argmax}{\operatornamewithlimits{argmax}} 
\newcommand{\argmin}{\operatornamewithlimits{argmin}} 
\newcounter{hypH}
\newenvironment{hypH}{\refstepcounter{hypH}\begin{itemize}
\item[{\bf A\arabic{hypH}}]}{\end{itemize}}

\newtheorem{theorem}{Theorem}[section]
\newtheorem{remark}[theorem]{Remark}
\newtheorem{proposition}[theorem]{Proposition}
\newtheorem{lemma}[theorem]{Lemma}

\title{Adaptive clustering by minimization of the mixing entropy criterion}

\author{ Thierry Dumont \thanks{This research has been conducted as part of the project Labex MME-DII (ANR11-LBX-0023-01).} \\
		MODAL'X, 
		UMR 9023, 
		UPL, 
		Univ. Paris-Nanterre, 
		F92000 Nanterre France\\
	\texttt{thierry.dumont@parisnanterre.fr} \\
}
\renewcommand{\shorttitle}{Mixing entropy clustering}
\renewcommand{\headeright}{}
\renewcommand{\undertitle}{}

\hypersetup{
pdftitle={Mixing entropy clustering},
pdfsubject={Clustering, Entropy, Mixture models},
pdfauthor={Thierry Dumont},
pdfkeywords={Entropy, Clustering, Mixture models, Order selection, EM algorithm, Expectation Maximization},
}

\begin{document}
\maketitle

\begin{abstract}
We present a clustering method and provide a theoretical analysis and an explanation to a phenomenon encountered in the applied statistical literature since the 1990's. This phenomenon is the natural adaptability of the order when using a clustering method derived from the famous EM algorithm. We define a new statistic, the relative entropic order, that represents the number of clumps in the target distribution. We prove in particular that the empirical version of this relative entropic order is consistent. Our approach is easy to implement and has a high potential of applications. Perspectives of this works are algorithmic and theoretical, with possible natural extensions to various cases such as dependent or multidimensional data.
\end{abstract}
\keywords{Entropy, Clustering, Mixture models, Order selection, EM algorithm, Expectation Maximization}

\section{Introduction} 
The present study follows on from the literature on model-based clustering. This research field in applied and theoretical statistics is very active since the 1990's (\citet{bryant:1991,celeux:1992,biernacki:2000,baudry:2012,celisse:2012,quost:2016,spurek:2017}). In the context of statistical data modeling using mixture distributions of some independent and identically distributed (i.i.d.) sample $(Z_1,\ldots,Z_n)$ with common probability distribution $P^\star$, model-based clustering pursue the three main objectives that are, 1/ Parameter inference when adjusting the data by a product measure $\prod_{k=1}^n \left(\sum_{x=1}^r \nu(x) g_{\theta_x}\left( Z_k\right) \right)$, 2/ Estimation of the mixture order $r$, 3/ Data clustering by computing, for instance, the maximum a posteriori estimators (MAP) $\widehat{X}_k = \argmax_x  \nu(x) g_{\theta_x}(Z_k)$.  While parameter inference is in general dealt with Expectation Maximization (EM) or gradient descent like algorithms (\citet{dempster:laird:rubin:1977}, \citet{baum:1970}), the order estimation is in general carried out using the model selection approach (\citet{akaike:1973},  \citet{mallows:1973}, \citet{massart:2007}) or using the famous Integrated Completed Likelihood method (\citet{biernacki:2000}) that performs the three tasks simultaneously. 

In this paper we present a clustering method as well as theoretical foundations that explain the behavior observed in some practical uses of a predecessor of the ICL: the Classification EM algorithm (\citet{celeux:1992},\citet{biernacki:1997}). We present a pure entropic based criterion that applies on any non parametric mixture decomposition $\sum_{x=1}^r \nu(x) G_x$  of $P^\star$ or of the empirical distribution $P^n$. It is made of the sum of two entropic terms:  the Shannon's entropy of $\nu$ and the weighted sum of the cross entropy of the $G_x$'s over a parametric probability density family $\{g_{\theta} ,\ \theta \in \Theta\}$ chosen beforehand.

The purpose of statistic inference is  to extract information from the data. Therefore, strong links exist between information theory (\citet{shannon:1948}) and statistics (see \citet{gassiat:2018}). In particular the maximum likelihood estimator (MLE) may also be seen as the minimum cross entropy estimator over a parametric family of models, that is the model that extracts the biggest quantity of information from the data. The entropy is a notion introduced by Claude Shannon for the information theory in his seminal work \citet{shannon:1948}. It measures how clumped up the probability measures are. Clumped up probability measures concentrate the total mass on a few zones. They are the most informative measures. On the contrary, spread out measures, meaning the measure with high entropy, present the most randomness and are the least informative. 

Our mixing entropy criterion, that we call the \textit{mixing entropy criterion}, realizes a compromise between the information contained in $\nu$, that favors the probabilities $\nu$  concentrated on a few $x$'s,  and the weighted sum of the cross entropy $G_x$'s over $\{g_{\theta} ,\ \theta \in \Theta\} $, that favors sharp mixture decomposition with spread out probabilities $\nu$'s. We show that this compromise leads to a natural decomposition of the distribution and  this decomposition is consistent.

We also show that our method  realizes a natural selection of the number of clusters $r$ and therefore prove the observation made in particular  by   \citet{biernacki:1997,spurek:2017} on numerical experiments. The classical model selection approach proceeds by penalization of a criterion by a term that reflects the model dimension, or its complexity in some sense, and it often relies on a manual calibration of this penalty using, for instance, the slope heuristic method (see \citet{baudry:2012}). On the opposite, classification by minimization of the  mixing entropy criterion selects a number of clusters without external calibration methods.  We prove that this number, that we call the \textit{relative entropic order}, is a statistic of the target distribution that is consistently estimated using its empirical version. This order represents the number of clumps in the distribution and the counting method is only relative to the chosen family of densities $\{g_{\theta} ,\ \theta \in \Theta\}$.

Section \ref{entrop:sec:entropic:approach} is the general section where the mixing entropy criterion is presented and where we prove that the minimum mixing entropy estimator converges when the order $r$ is kept fixed.  Section \ref{entrop:sec:discussion:Dr} is devoted to a discussion on the nature of the possible limits of this estimator. It is a transitional section where applications to the Gaussian and the binary settings are detailed. This section prepares the definition of the relative entropic order defined in Section \ref{entrop:sec:entropic:order}. Consistency of the empirical relative entropic order is proven in the same section. In Section \ref{entrop:sec:simi:mixture:models} we observe that the mixing entropy criterion is a quantity that notably appears when dealing with classical mixture models. We also show in this section that the CEM algorithm creates a sequence of decompositions   with non-increasing mixing entropy. Section \ref{entrop:sec:practical:implement} leans on the preceding observations to build an algorithm that we use in Section \ref{entrop:sec:numerical} to illustrate, on synthetic data, the results of this paper. 
Note that if the distribution $P^\star$ is itself a mixture distribution of order $r$, then the  relative entropic order may not be equal to $r$. both orders do not measure the same quantity.  We  illustrate this phenomena on the numerical experiments of Section \ref{entrop:sec:numerical}.  Finally, some detailed proofs are gathered in Section \ref{entrop:sec:proofs} and in the Supplementary material \citet{dumont:supp:2022}.

\section{Main setting and mixing entropy}
\label{entrop:sec:entropic:approach}
Throughout the paper we consider a probability space $(\Omega,\mathcal{F},\mathbb{P})$. Let $\mathbb{Z}$ be a topological space equipped with its Borel $\sigma$-algebra. We also consider some non-negative reference measure $\lambda$ on $\mathbb{Z}$.
\subsection{Basic definitions and general assumptions}
\label{entrop:sec:main}
Let $P$ be a probability distribution on $\mathbb{Z}$. 
If  $P$ is  relatively continuous with respect to $\lambda$: $\mathrm{d}P( z) = p(z)\rm d\lambda\left(  z\right)$, the Shannon's entropy of $P$  is defined as
\begin{equation*}
H(P) := - \int p(z)\log p(z) \rm d \lambda(z)\;.
\end{equation*}
The cross entropy  between $P$  and a function $g$ on $\mathbb{Z}$, positive $P$-almost surely (\textit{a.s.}) and $\log$-integrable  is:  
\begin{equation*}
H(P\ ||\ g )  := - \mathbb{E}_P\left(\log  g\right) := -\int \log g(z) \mathrm{d}P\left( z\right)\;.
\end{equation*} 
In the case where $P$ is  relatively continuous with respect to $\lambda$, and if $g$ is a probability density with respect with the same measure $\lambda$, then $H(P\ ||\ g ) $ satisfies  
$$H(P\ ||\ g )  = -\int p(z)\log(g(z)) \mathrm{d}\lambda\left( z\right)\;,$$ 
and $KL(P\ || g) := H(P\ ||\ g ) -H(P)$  is known as the \textit{Kullback-Leibler} divergence, also called the relative entropy,  between $p$ and $g$ (see \citet{kullback:1997}).

In the context of inference, it is common knowledge that a nice interpretation of the classical maximum likelihood estimator (MLE) is to see the estimator as a minimizer of the relative entropy (or the Kullback-Leibler divergence):

Let $ \left( Z_1,\ldots,Z_n\right)$ be a vector of independent and identically distributed (\textit{i.i.d.}) variables on $\mathbb{Z}$ with common distribution $P^\star$,  $\left\{g_\theta \;,\; \theta \in \Theta\right\}$ be a family of densities on $\mathbb{Z}$ with respect with $\lambda$  and  $\ell$ be the $\log$-likelihood function, defined, for all $\theta$ in $\Theta$ by: 

$$\ell(\theta) = \frac{1}{n} \sum_ {k=1}^n \log(g_\theta(Z_k))  \;. $$
Denoting by $P^n$ the empirical distribution of  $ \left( Z_1,\ldots,Z_n\right)$, then $\ell(\theta) = - H(P^n\ ||\ g_\theta )$ and a maximizer $\widehat{\theta}_n$ of $\ell$, if it exists, is also a minimizer of the cross entropy and therefore of the Kullback-Leibler divergence $\theta \mapsto KL(P^n\ ||\ g_\theta )$.
If the underlying distribution satisfies $\rm dP_\star(z) = g_{\theta^\star}(z) \rm d \lambda
(z)$, for some $\theta^\star$ in $\Theta$,  then $\min_{\theta} KL(P^\star\ ||\ g_{\theta} )= KL(P^\star\ ||\ g_{\theta^\star} ) = 0$.
Moreover, the law of large numbers insures that $H(P^n\ ||\ g_\theta )$ converges, as $n$ grows to $\infty$, towards $H(P^\star\ ||\ g_\theta ) $. These arguments,  together with continuity, compacity and identifiabily assumptions on the model, lead to the consistency of the MLE in a large variety of frameworks.  

We now embrace this entropic point of view and build a mixing version of the criterion.
Let $r$ be a positive integer. Denote by $\mathbb{X}$ the set $\{1,\ldots,r\}$, and by $\mathcal{M}_1\left( \mathbb{X}\right)$ the set of probability vectors $\nu =( \nu(1),\ldots,\nu(r))$ satisfying, for all $x$ in $\mathbb{X}$, $\nu(x) \ge 0$ and $\sum_{x=1}^r\nu(x)=1$. In the sequel, we indifferently use the notation $H$ for the entropy of a density in $\mathbb{X}$ or in $\mathbb{Z}$. Therefore, for any $\nu \in \mathcal{M}_1\left( \mathbb{X}\right)$,  $H(\nu) = \sum_{x=1}^r \nu(x) \log(\nu(x)) $. Note that we use the classical convention $0\log(0)=0$. We also denote by $\mathcal{M}_1\left( \mathbb{Z}\right)$ the set of all probability distributions in $\mathbb{Z}$. 

Let $\Theta$ be a parameter set and  $\{g_\theta\;,\;\theta \in \Theta\}$ be a family of probability density functions  relatively to a non negative reference measure $\lambda$ on $\mathbb{Z}$. In this section and in Section  \ref{entrop:sec:entropic:order}, we will consider the following assumptions on $\mathbb{Z}$, $\Theta$ and $\{g_\theta\;,\;\theta \in \Theta\}$:

\begin{hypH}
\label{entrop:hyp:theta:compact}
 $\Theta$ is a non empty compact topological space. 
\end{hypH}

\begin{hypH}
\label{entrop:hyp:density:bounded}
There exists a constant $C>1$ such that, for all $\theta \in \Theta$, $g_\theta$ is continuous and, for all $z$ in $\mathbb{Z}$, $1/C\le  g_\theta(z) \le C$ 
\end{hypH}

We denote by $\mathcal{C}_{b,C}\left(\mathbb{Z} \right)$ the set of all continuous upper bounded by $C$ and lower bounded by $1/C$, and we equip this space with the topology of the uniform convergence. 

\begin{hypH}
\label{entrop:hyp:density:continuous}
The application $\theta \mapsto g_\theta$ from $\Theta$ to  $\mathcal{C}_{b,C}\left(\mathbb{Z} \right)$ is continuous.
\end{hypH}
Finally we make the following assumption on the observation space $\mathbb{Z}$. 
\begin{hypH}
\label{entrop:hyp:Z:compact}
$\mathbb{Z}$ is a compact metric space.
\end{hypH}

\begin{remark}
\label{entrop:rem:gaussian:setting}
\begin{enumerate}
\item In the paper we will illustrate our results by considering $\mathbb{Z} = \mathbb{R}$ and $g_{\theta}(z)  = \frac{1}{\sqrt{2 \pi \sigma^2}}\exp \left( - \frac{(z-\mu)^2}{2\sigma^2}\right)$, $\theta = (\mu,\sigma^2) \in \Theta = \mathbb{R}\times ]0,+\infty[$, despite the fact that this choice does not satisfy Assumptions A\ref{entrop:hyp:theta:compact}, A\ref{entrop:hyp:density:bounded} and A\ref{entrop:hyp:Z:compact}. This choice provides a better understanding of the illustrated notions since the Gaussian mixture is the classical mixture setting. Moreover, while   Assumptions A\ref{entrop:hyp:theta:compact}, A\ref{entrop:hyp:density:bounded} and A\ref{entrop:hyp:Z:compact} are used to ease the proofs, one could project that these assumptions could be weakened, in particular for the Gaussian setting since the simulations seem to illustrate our results in that specific case. 
\item Assumption A\ref{entrop:hyp:Z:compact} is a strong assumption, nevertheless it implies, thanks to the Riesz representation theorem, the compactness of $\mathcal{M}_1\left(\mathbb{Z} \right)$ stated in Proposition \ref{entrop:prop:probas:compacness} below. This result is commonly known as the Banach-Alaoglu theorem (see \citet{rudin:1991}).
Adding tightness assumptions on $\left\{g_\theta \;,\; \theta \in \Theta\right\}$ and $P^\star$  could allow us to weaken   A\ref{entrop:hyp:Z:compact} by assuming that $\mathbb{Z}$  is locally compact only.  
\end{enumerate}
\end{remark}

\begin{proposition}
\label{entrop:prop:probas:compacness}
$\mathcal{M}_1\left(\mathbb{X} \right)$ and, if Assumption A\ref{entrop:hyp:Z:compact} holds, $\mathcal{M}_1\left(\mathbb{\mathbb{Z}} \right)$ and therefore $\mathcal{M}_1\left(\mathbb{X} \right)\times\mathcal{M}_1\left( \mathbb{Z}\right)^r$ are compact sets relatively to their weak$^\star$ topology - that is the topology of the simple convergence over the continuous functions.
\end{proposition}

In the sequel we use the following notation: 
\begin{equation}
\label{entrop:eq:Dr}
\mathcal{D}_r = \mathcal{M}_1\left(\mathbb{X} \right)\times \mathcal{M}_1\left( \mathbb{Z}\right)^r
\end{equation} 

\subsection{Mixing entropy criterion}

For all $\theta = (\theta_1,\ldots,\theta_r) =: (\theta_x)_{x=1}^r $ in $\Theta^r$, we define the applications 
\begin{equation*}
\begin{array}{cccc}
\mathbb{H}_\theta \ : & \mathcal{D}_r  &\longrightarrow& \mathbb{R}\\
&(\nu,\left(G_x \right)_{x=1}^r))  &\longmapsto& H(\nu) + \sum_{x=1}^r \nu(x) H\left(G_x  ||g_{\theta_x}\right)
      \end{array}
\end{equation*}

 and 

\begin{equation*}
\begin{array}{cccc}
\mathbb{H}\ : & \mathcal{D}_r  &\longrightarrow& \mathbb{R}\\
&(\nu,\left(G_x \right)_{x=1}^r))  &\longmapsto&  \inf_{\theta \in \Theta^r}\mathbb{H}_\theta(\nu,\left(G_x \right)_{x=1}^r) 
      \end{array}
\end{equation*}

\begin{remark}
\label{entrop:rem:hiddenrdep}
\begin{enumerate}
\item We call the functions $\mathbb{H}_\theta$ and $\mathbb{H}$  \textit{mixing entropy functions} or \textit{criteria}.
\item Despite the fact that $r$ will vary, we voluntarily omit to indicate the dependency in $r$ of the  mixing entropy  functions.  It is justified since, for any $r'>r\ge 1$, we may embed any vector $\left(\nu, \left( G_x\right)_{x=1}^r \right)$ in $\left(\mathcal{C}^\star_{r'} \right)$ while keeping its mixing entropy: Define, for $x$ in $\{1,\ldots,r\}$,  $\widetilde\nu(x)=\nu(x)$ and, for $r'\ge x>r$, $\widetilde\nu(x)=0$. Define,  for $x$ in $\{1,\ldots,r\}$, $\widetilde G_x=G_x$ and, for $r'\ge x>r$, define $\widetilde G_x$ as any element of $\mathcal{M}_1\left( \mathbb{Z}\right)$. Then, thanks to the convention $0\log 0= 0$,  $$\mathbb{H}(\nu,\left(G_x \right)_{x=1}^r)) = \mathbb{H} \left(\widetilde\nu,\left(\widetilde G_x \right)_{x=1}^{r'}\right)\;.$$
\end{enumerate}

\end{remark}

Proposition \ref{entrop:prop:H:continue} below states the existence and the continuity of the mixing entropy functions under the contions  A\ref{entrop:hyp:theta:compact}-\ref{entrop:hyp:density:continuous}.
\begin{proposition}
\label{entrop:prop:H:continue}
Under A\ref{entrop:hyp:theta:compact}-\ref{entrop:hyp:density:continuous}, 
\begin{enumerate}
\item for all $\theta$ in $\Theta$ $\mathbb{H}_\theta$ is well defined ($\mathbb{H}_\theta<\infty$) ,\label{entrop:eq:prop:Htheta}
\item $\mathbb{H}$ is well define ( $\mathbb{H}<\infty$) and satisfies, for all $(\nu,\left(G_x \right)_{x=1}^r))$ in $ \mathcal{D}_r$, \label{entrop:eq:prop:H}
\begin{equation*}
\mathbb{H}(\nu,\left(G_x \right)_{x=1}^r)) = H(\nu) + \sum_{x=1}^r \nu(x) \inf_{\theta_x \in \Theta}H\left(G_x  ||g_{\theta_x}\right)\;,
\end{equation*} 
\item the functions  $\mathbb{H}_\theta$, for all $\theta$ in $\Theta$, and  $\mathbb{H}$ are continuous, \label{entrop:eq:prop:continuity}
\item and the infimums $\inf_{\theta_x \in \Theta}H\left(G_x  ||g_{\theta_x}\right)$ are reached in $\Theta$.\label{entrop:eq:prop:Hbis}
\end{enumerate}
\end{proposition}  
\begin{proof} 
For all $G$ in $\mathcal{M}_1\left( \mathbb{Z} \right)$, and all $\theta$ in $\Theta$, $H(G\ ||\ g_{\theta}) = - \int_\mathbb{Z} \log(g_\theta(z)) \mathrm{d}G(z)$ and, by A\ref{entrop:hyp:density:bounded}, $$|\log(g_\theta(z))|\le \log(C)\;,$$ proving points \ref{entrop:eq:prop:Htheta} and \ref{entrop:eq:prop:H}. The same argument proves, by definition of the weak* topology, that, for any $\theta$ in $\Theta$, $G \mapsto H(G\ ||\ g_{\theta})$ is continuous and it is straightforward to show that $\nu \mapsto H(\nu) = - \sum_{x=1}^r \nu(x) \log(\nu(x))$ is also continuous. Thus, for all $\theta$ in $\Theta^r$, $\mathbb{H}_\theta$ is continuous. Finally, this last point together with the compactness assumption on $\Theta$ (A\ref{entrop:hyp:theta:compact}) lead to the continuity of $\mathbb{H}$ and achieves the proof of points \ref{entrop:eq:prop:continuity} and \ref{entrop:eq:prop:Hbis}.
\end{proof}

Now, define the following subsets/condition on $\mathcal{D}_r$, 

\begin{eqnarray*}
\left(\mathcal{C}^\star_r \right)& =& \left\{(\nu,\left(G_x \right)_{x=1}^r)) \in  \mathcal{D}_r  \ \mbox{such that} \ \sum_{x=1}^r \nu(x) G_x = P^\star \right\}\;,\\
\left(\mathcal{C}^n_r \right)& =& \left\{(\nu,\left(G_x \right)_{x=1}^r)) \in  \mathcal{D}_r  \ \mbox{such that} \ \sum_{x=1}^r \nu(x) G_x = P^n \right\}\;.
\end{eqnarray*}
 
$\left(\mathcal{C}^\star_r \right)$ (resp. $\left(\mathcal{C}^\star_r \right)$) is necessarily non empty since it contains $\left(1, \left(P^\star \right) \right)$  (resp. $\left(1, \left(P^n \right) \right)$). It is made of all possible mixture decompositions of $P^\star$ (resp. of $P^n$) into $r$ distributions. If A\ref{entrop:hyp:Z:compact} holds, then  $\left(\mathcal{C}^\star_r \right)$ and $\left(\mathcal{C}^\star_r \right)$  are compact subsets of $\mathcal{D}_r$.
\begin{remark}
\label{entrop:rem:C:r1}
If $r=1$, $\left(\mathcal{C}^\star_1 \right)$ (resp. $\left(\mathcal{C}^n_1 \right)$ ) is made of the single element $\left( 1,(P^\star)\right)$ (resp. $\left( 1,(P^n)\right)$).
\end{remark}

\begin{proposition}
\label{entrop:prop:abscont}
Let   $r\ge 1$, and let $\left( \nu, \left( G_x\right)_{x=1}^r\right)$ be in $\left(\mathcal{C}^\star_r\right)$ (resp. $\left(\mathcal{C}^n_r\right)$). For all   $x$ in $\{1,\ldots,r\}$ such that $\nu(x)>0$, $G_x$ is absolutely continuous with respect with $P^\star$ (resp. $P^n$).
\end{proposition}
\begin{proof} Let $P$ be equal to $P^\star$ or $P^n$. Let  $(\nu,\left(G_x \right)_{x=1}^r)$ be a mixture decomposition of $P$.
If $P(A) = 0$ for some measurable set $A$, then $\sum_{x=1}^r \nu(x) G_x(A) =\sum_{x\ | \ \nu(x)>0} \nu(x) G_x(A) = 0$. Then $G_x(A) = 0 $ for all $x$ such that $\nu(x)>0$.

\end{proof}

Define the sets
\begin{eqnarray}
\mathcal{D}^\star_r &:=& \left\{(\nu,\left(G_x \right)_{x=1}^r)) \in \left(\mathcal{C}^\star_r \right) \mbox{ such that } \mathbb{H} (\nu,\left(G_x \right)_{x=1}^r)) = \inf_{\left(\mathcal{C}^\star_r \right)} \mathbb{H}\right\}\;,\label{entrop:eq:Dstar:definitions} \\
\mathcal{D}^n_r &:=& \left\{(\nu,\left(G_x \right)_{x=1}^r)) \in \left(\mathcal{C}^n_r \right) \mbox{ such that } \mathbb{H} (\nu,\left(G_x \right)_{x=1}^r)) = \inf_{\left(\mathcal{C}^n_r \right)} \mathbb{H}\right\}\;.\label{entrop:eq:Dn:definitions} 
\end{eqnarray}

\begin{remark}
\label{entrop:rem:label:permutation}
For all $\left( \nu , \left( G_x\right)_{x=1}^r\right)$ and all $\sigma$ permutation of $\left\{1,\ldots,r \right\}$ - we call $\sigma$ a labels permutation - if $\left( \nu , \left( G_x\right)_{x=1}^r\right)$ belongs to $\left(\mathcal{C}^\star_r \right)$ (resp. $\left(\mathcal{C}^n_r \right)$), then  $\left( \nu \circ \sigma , \left( G_{\sigma(x)}\right)_{x=1}^r\right)$ also belongs to $\left(\mathcal{C}^\star_r \right)$ (resp. $\left(\mathcal{C}^n_r \right)$). Moreover, it is straightforward to see that $\mathbb{H}$ is invariant under labels permutation and the same result holds for $\mathcal{D}^\star_r$ and $\mathcal{D}^\star_n$.
\end{remark}
As a straightforward consequence of the  continuity of  $\mathbb{H}$ in  Proposition \ref{entrop:prop:H:continue}, Proposition \ref{entrop:prop:D:definitions} below holds: 

\begin{proposition}
\label{entrop:prop:D:definitions}
Assume A\ref{entrop:hyp:theta:compact}-\ref{entrop:hyp:Z:compact},  then for all $r\ge 1$,  $\inf_{\left(\mathcal{C}^\star_r \right)} \mathbb{H}$ (resp. $\inf_{\left(\mathcal{C}^n_r \right)} \mathbb{H}$ ) is reached in $\left(\mathcal{C}^\star_r \right)$ (resp. in $\left(\mathcal{C}^n_r \right)$) and, consequently, $\mathcal{D}^\star_r$  and  $\mathcal{D}^n_r$ are non empty.
\end{proposition}

 $\mathcal{D}^\star_r$  (resp. $\mathcal{D}^n_r$) is made of the mixture decompositions of  $P^\star$ (resp. $P^n$) that minimize the mixing entropy $\mathbb{H}$ which is the \textit{best} compromise between the entropy of $\nu$ and the average cross entropy between the distributions $G_x$'s and the family $\left\{g_\theta \ , \ \theta \in \Theta \right\}$.
The first remarkable result is given by Theorem \ref{entrop:th:Dn:converge} below that ensures the consistency of the  optimal mixture decompositions of $P^n$.
 
\begin{theorem}
\label{entrop:th:Dn:converge}
Assume A\ref{entrop:hyp:theta:compact}-\ref{entrop:hyp:Z:compact},  then for all $r\ge 1$,
$$\mbox{a.s.} \;,\;\lim\limits_{n \to \infty} \mathcal{D}^n_r  \subset \mathcal{D}^\star_r \;,  $$ 
meaning that if we choose,  for all $n\ge 1$,  $(\nu^n,\left(G^n_x \right)_{x=1}^r))  $ in  $\mathcal{D}^n_r $, then any convergent subsequence  $(\nu^{u_n},\left(G^{u_n}_x \right)_{x=1}^r))_{n\ge 1} $ in the compact set $\mathcal{D}_r$ \eqref{entrop:eq:Dr} has its limit in $\mathcal{D}^\star_r $.
\end{theorem}
\begin{proof}
We start the proof with  Lemma \ref{entrop:lemma:aposteriori:X} that allows,  in the context of a mixture distribution,  to build  the hidden variable posteriorly on the observation. Let $(\nu^\star,\left(G^\star_x \right)_{x=1}^r)$ be any decomposition of $P^\star$ in  $\left( \mathcal{C}^\star_r\right)$. 
Let  $\left(\widetilde X, \widetilde Z \right)$ be a random vector where  $\widetilde X$ is distributed according to $\nu^\star$ and, conditionally on $\widetilde X = x$, $\widetilde Z$ is distributed according to $G^\star_x$. Define, for all  $z$ in the support of $P^\star$, and all $x$ in $\mathbb{X}$, 
\begin{equation}
\label{entrop:eq:phi:x:cond}
\Phi(x|z) = \mathbb{P} \left(\widetilde X =x |\widetilde Z= z  \right)\;.
\end{equation}
\begin{lemma}
\label{entrop:lemma:aposteriori:X}
If $Z$ is a random variable distributed according to $P^\star$. If, conditionnaly on $Z$, $X$ is distributed according to  $\Phi(\cdot|Z)$, defined by \eqref{entrop:eq:phi:x:cond}, then 
$(X,Z)$ is distributed according to  the joint distribution $P^\star(x,\rm d z) = \nu^\star(x) G^\star_x(\rm d z)$. 
\end{lemma}
The proof of Lemma \ref{entrop:lemma:aposteriori:X} is straightforward. Now, let $(\nu^{u_n},\left(G^{u_n}_x \right)_{x=1}^r)_{n\ge 1} $ such as in the statement of Theorem \ref{entrop:th:Dn:converge}. Denote by $(\nu^{\infty},\left(G^{\infty}_x \right)_{x=1}^r) $  the limit, in $\mathcal{D}_r$, of this subsequence. By the law of large number,  $(\nu^{\infty},\left(G^{\infty}_x \right)_{x=1}^r) $ belongs to $\left(\mathcal{C}^\star_r \right)$.
Now, let $(\nu^\star,\left(G^\star_x \right)_{x=1}^r))  $ be any element of $\mathcal{D}^\star_r$, and let, for all $k$ in $\{1,\ldots,n\}$, $\Phi_k(x) = \Phi(x|Z_k)$ and $X_k$ a random variable distributed according to $\Phi_k$ such as described in Lemma \ref{entrop:lemma:aposteriori:X}. Define $\nu^\star_n(x) = \frac{1}{n} \sum_{k=1}^{n}\mathds{1}_{x} (X_k)$ and  
$$G^\star_{n,x} = \frac{ \sum_{k=1}^n  \mathds{1}_{x}(X_k) \delta_{Z_k}  }{\nu^\star_n(x) }\mbox{ if } \nu^\star_n(x)\neq 0, \mbox{ and }P^\star \mbox{ otherwise}\;.$$
where $\delta_{Z_k}$ is the Dirac distribution on $\{Z_k\}$. Then $\left(\nu^\star_n , \left(G^\star_{n,x} \right)_{x=1}^r \right)$
belongs to  $\left( \mathcal{C}^n_r\right)$  and satisfies, a.s., $\nu^\star_n\xrightarrow[n\to\infty]{\mbox{weak}^\star}\nu^\star$  and, for all $x$ in $\{1,\ldots,r\}$, a.s., $G^\star_{n,x}\xrightarrow[n\to\infty]{\mbox{weak}^\star}G^\star_{x}$ . By Proposition \ref{entrop:prop:H:continue}, $\mathbb{H}$ is continuous and $\lim\limits_{n \to \infty}\mathbb{H}\left(\nu^\star_n,(G^\star_{n,x})_{x=1}^r\right) = \mathbb{H}\left(\nu^\star,(G^\star_{x})_{x=1}^r\right)$.
Moreover, by definition of $(\nu^{u_n},\left(G^{u_n}_x \right)_{x=1}^r)_{n\ge 1} $, for all $n$,
$$\mathbb{H}\left(\nu^\star_{u_n},(G^\star_{{u_n},x})_{x=1}^r\right)\ge\mathbb{H}\left( \nu^{u_n},\left(G^{u_n}_x \right)_{x=1}^r\right)\;,$$
which leads, when $n$ tends to $\infty$, to 
$$\inf_{\left(\mathcal{C}^\star_r \right)}\mathbb{H} = \mathbb{H}\left(\nu^\star,(G^\star_{x})_{x=1}^r\right)\ge \mathbb{H}\left(\nu^\infty,(G^\infty_{x})_{x=1}^r\right)\;.$$ 
\end{proof}
\section{Properties of $\mathcal{D}^\star_r$} 
\label{entrop:sec:discussion:Dr}
\subsection{Interpretation of  $\mathcal{D}^\star_r$ as a classification rule}
Throughout this section we assume  that the following assumption holds:

\begin{hypH}
\label{entrop:hyp:log:g:integrable}
For all $\theta$ in $\Theta$, $$\int_{\mathbb{Z}} \left|\log(g_{\theta}(z))\right| \mathrm{d}P^\star(z) <\infty\;.$$
\end{hypH}

Let $r\ge 1$. Let $\left(\nu,\left( G^x\right)_{x=1}^r \right)$ be any element in $\left( \mathcal{C}^\star_r \right)$. From  Proposition \ref{entrop:prop:abscont}, for all $x$ in $\{1,\ldots,r\}$, there exists $g_x$ in $L^1\left(P^\star\right)$   such that $\rmd G_x(z) = g_x(z) \rmd P^\star(z) $. Define, 
\begin{equation}
\label{entrop:eq:phix}
\phi_x =  \nu(x) g_x\;.
\end{equation}
Let $\Phi_r\left(\mathbb{Z} \right)$ be the set of all the functional vectors $\left(\phi_x \right)_{x=1}^r$  such that for all $x$ in $\{1,\ldots,r\}$, $\phi_x$ is a measurable, $[0,1]$-valued, function of $\mathbb{Z}$ satisfying, for all $z$ in $\mathbb{Z}$, $\sum_{x=1}^r \phi_x (z)= 1$. If the  $\phi_x$'s  are defined by \eqref{entrop:eq:phix}, then $\left(\phi_x \right)_{x=1}^r$ belongs to $\Phi_r\left(\mathbb{Z} \right)$ even if that means changing the $g_x$'s on a $P^\star$-negligible set. Conversely, for all  $\left(\phi_x \right)_{x=1}^r$ in $\Phi_r\left(\mathbb{Z} \right)$, for all $x$ in $\{1,\ldots,r\}$, define  $\nu^\phi(x)$ and $G^\phi_x$  by :
\begin{align}
\label{entrop:eq:nux}
\nu^\phi(x) &:= \int_\mathbb{Z}\phi_x(z)  \rmd P^\star(z) \;,\\ 
 \rm d G^\phi_x(z)  &:= \frac{1}{\nu^\phi(x)} \phi_x(z)  \rmd P^\star(z) \;,
 \label{entrop:eq:Gx}
\end{align}
where \eqref{entrop:eq:Gx} only applies if $\nu^\phi(x)$, given by \eqref{entrop:eq:nux}, is positive (otherwise set $ G^\phi_x$ to any distribution in $\mathcal{M}_1(\mathbb{Z})$). Define for all $\phi$ in $\Phi_r\left(\mathbb{Z} \right)$ the mixing entropy of $\phi$ relatively to $P^\star$: 
 
\begin{equation}
\label{entrop:eq:H:phi}
\mathbb{H}_{P^\star}(\phi):= \mathbb{H}\left(\nu^\phi,\left( G^\phi_x\right)_{x=1}^r\right)\;.
\end{equation}
Then, a basic manipulation of  \eqref{entrop:eq:H:phi} shows that, if we define, for all $\theta =(\theta_1,\ldots,\theta_r)$,
$$\mathbb{H}_{p^\star}(\phi,\theta):= - \sum_{x=1}^r \int_{\mathbb{Z}} \log \left[ g_{\theta_x}(z) \nu^\phi(x) \right] \phi_x(z) \rm d P^\star(z) \;,$$ 
then
\begin{equation*} 
\mathbb{H}_{p^\star}(\phi)= \inf_{\theta \in \Theta^r} \mathbb{H}_{p^\star}(\phi,\theta).
\end{equation*}
Moreover, we necessarily have
\begin{equation}
\label{entrop:eq:parallel:phi:nuG}
\inf_{ \Phi_r\left(\mathbb{Z} \right)}\mathbb{H}_{p^\star}(\phi) = \inf_{  \left(\mathcal{C}^\star_r \right) } \mathbb{H}\left(\nu ,\left( G_x\right)_{x=1}^r\right)\;.
\end{equation}
Define
\begin{equation}
\label{entrop:eq:phi:star:r}
\Phi^\star_r  := \left\{\phi \in \Phi_r\left(\mathbb{Z} \right) \ | \ \mathbb{H}_{p^\star}(\phi) = \inf_{\Phi_r\left(\mathbb{Z} \right)}\mathbb{H}_{p^\star}(\phi)  \right\}\;,
\end{equation}
then Proposition \ref{entrop:prop:Dstar} below is straightforward:
\begin{proposition}
\label{entrop:prop:Dstar} 
If $\Phi^\star_r$ is not empty, then 
$$\mathcal{D}^\star_r = \left\{\left( \nu^\phi,\left( G^\phi_x\right)_{x=1}^r \right)  \ | \ \phi \in  \Phi^\star_r \right\}\;.$$
\end{proposition}
Now consider the following assumptions: 
\begin{hypH}
\label{entrop:hyp:Z:metric}
The topology of $\mathbb{Z}$ is induced by a metric $d$.
\end{hypH}

If A\ref{entrop:hyp:Z:metric} holds, we denote by $B_d(z_0,\varepsilon)$ the open ball in $\mathbb{Z}$ with respect to $d$, centered in $z_0$ and of radius $\varepsilon$. In that framework,  we define the support of the Borel measure $P^\star$ as the set of all $z$ in $\mathbb{Z}$ such that for all $\varepsilon>0$, $P^\star(B_d(z,\varepsilon))>0$. 

\begin{hypH}
\label{entrop:hyp:g:continuous}
For all $\theta$ in $\Theta$, $g_{\theta} $ is a continuous positive function of $\mathbb{Z}$.
\end{hypH}

\begin{hypH}
\label{entrop:hyp:openZ:infinite}
There exists an  open subset $U$, containing an infinite number of elements of $\mathbb{Z}$, that is included in the support of $P^\star$.
\end{hypH}

\begin{hypH} 
\label{entrop:hyp:theta:coincident}
For all $\theta_1$ and $\theta_2$ in $\Theta$, if for some constant $K>0$, $g_{\theta_1}(z) =  K g_{\theta_2}(z)$ for an infinite number of $z$'s in $\mathbb{Z}$ then, necessarily $g_{\theta_1}=g_{\theta_2}$.
\end{hypH}

\begin{remark}
\label{entrop:rem:expofamily}
 Using  Definition 10.1.5 of \cite{cappe:moulines:ryden:2005}. Consider the case where the family $\{g_\theta\;,\; \theta\in\Theta\}$ is an exponential family of $\mathbb{Z}$ which is: for all $\theta$ in $\Theta$ and all $z$ in $\mathbb{Z}$,
$$g_\theta(z) = h(z)\exp \left( \psi(\theta)^tS(z) - c(\theta)\right)\;,$$
where $S$ (known as as the vector of natural sufficient statistics) and $\psi$ are vector valued functions of the same dimension on $\mathbb{Z}$ and $\Theta$ respectively, $c$ is a real-valued function on $\Theta$ and $h$ is a non-negative real valued function on $\mathbb{Z}$.  Then A\ref{entrop:hyp:theta:coincident} is equivalent to:  $z\mapsto \psi^t S(z) $
constant for an infinite number of $z$'s i.i.f. $\psi = 0$. It is the case in the Gaussian setting where $\mathbb{Z} = \mathbb{R}$ and $S(z) = (1,z,z^2)$:   $\psi_0 + \psi_1z + \psi_2z^2 = 0$ for more than three $z$' i.i.f. $\psi_0=\psi_1=\psi_2=0$. 
\end{remark}

 Alternatively to A\ref{entrop:hyp:Z:metric}-\ref{entrop:hyp:theta:coincident}, we will consider the following assumption:
\begin{hypH}
\label{entrop:hyp:ident:discret}
 $P^\star$ is a discrete distribution on $\mathbb{Z}$. 
\end{hypH}

\begin{theorem} 
\label{entrop:th:phistar} 
Assume A\ref{entrop:hyp:log:g:integrable}.  Assume A\ref{entrop:hyp:Z:metric}-\ref{entrop:hyp:theta:coincident} or A\ref{entrop:hyp:ident:discret}. For all $r\ge 1$, all  $\phi^\star$ in  $\Phi_r\left(\mathbb{Z} \right)$ and all $\theta^\star$ in $\Theta^r$, if  there exists $x_0$ in $\{1,\ldots,r\}$ such that  $P^\star\left(\phi^\star_{x_0} \not\in \{0,1\} \right)>0$, then there exists $\phi$ in  $\Phi_r\left(\mathbb{Z} \right)$ such that: $$\mathbb{H}\left(\phi,\theta^\star \right)<\mathbb{H}\left(\phi^\star,\theta^\star \right) \;.$$
Therefore, if  $\phi^\star$ belongs to $\Phi^\star_r $, then, for all $x$ in $\{1,\ldots,r\}$,  $P^\star$-a.s., $\phi^\star_x= 0$ or $\phi^\star_x = 1$.  Equivalently, for any  mixture decomposition $\left( \nu,\left( G_x\right)_{x=1}^r\right)$ in $\mathcal{D}^\star_r$, the $G_x$'s are necessarily singular.
\end{theorem} 

The proof of Theorem \ref{entrop:th:phistar} is postponed in Section  \ref{entrop:sec:proof:th:phistar}.

\begin{remark}
\label{entrop:rem:classif}
Define the empirical version of $\Phi^\star_r$ (Equation \eqref{entrop:eq:phi:star:r}),
\begin{equation}
\label{entrop:eq:phi:n:r}
\Phi^n_r  := \left\{\phi \in \Phi_r (\{Z_1,\ldots,Z_n\})  \ | \ \mathbb{H}_{P^n}(\phi) = \min_{ \Phi_r (\{Z_1,\ldots,Z_n\})}\mathbb{H}_{P^n} \right\}\;,
\end{equation} 
then $P^n$ satisfies A\ref{entrop:hyp:ident:discret} and Theorem \ref{entrop:th:phistar} applies to $P^n$: for all $\phi$ in $\Phi^n_r $, for all $x$ in $\{1,\ldots,r\}$ and all $z$ in  $\{Z_1,\ldots, Z_n\}$, $\phi_x(z)$ equals $1$ or $0$. Moreover, $\phi$ belongs to $\Phi_r\left(\{ Z_1,\ldots,Z_n\}\right)$, thus  $\sum_{x=1}^r\phi_x(z) = 1$ and there exists exactly one $x$  in $\{1,\ldots,r\}$ such that $\phi_x(z)=1$.  We can therefore define, for all $k$ in $\{1,\ldots,n\}$,  $x^\phi_k$ as the unique $x$ in $\{1,\ldots,r\}$ such that $\phi_x(Z_k)=1$. 
Thus, the determination of $\Phi^n_r$ consists in finding the assignment  $(X_1,\ldots,X_n)$ in $\left\{1,\ldots,r \right\}^n$ (the classification rule) that minimizes the mixing entropy criterion.
\end{remark}

\subsection{Mixing entropy in the Gaussian mixture case}
 
We focus  here on the case where  $\left\{g_\theta \;,\; \theta \in \Theta\right\}$ is the Gaussian density family (c.f. Remark \ref{entrop:rem:gaussian:setting}). Define, for all  $\theta = (\mu,\sigma^2)$ in $\mathbb{R}\times ]0,\infty[$, $$g_\theta(z) = \frac{1}{2\pi \sigma^2} \exp\left( -\frac{(z-\mu )^2}{2\sigma^2}\right)\;.$$ We also assume that there exist $\left(\nu^\star_1,\nu^\star_2 \right)$ in $\mathcal{M}\left( \left\{1,2 \right\}\right)$, $\theta^\star_1=\left(\mu^\star_1,{\sigma^\star_1}^2 \right)$ and $\theta^\star_2=\left(\mu^\star_2,{\sigma^\star_2}^2 \right)$ such that $P^\star(\rm d z) = p^\star(z)\rm d z$, with 

\begin{equation}
\label{entrop:eq:gaussian:pstar}
p^\star = \nu^\star_1 g_{\theta^\star_1} + \nu^\star_2 g_{\theta^\star_2}
\end{equation} 

\begin{remark}
Note that, using Remark \ref{entrop:rem:expofamily}, Theorem \ref{entrop:th:phistar} applies and, if 

$\left((\nu(1),\nu(2)),\left( G_1,G_2\right)\right)$ belongs to $ \mathcal{D}^\star_2$, with $\nu(1),\nu(2) \not\in\{0,1\}$, then, for $x=1,2$, there exist $\phi_x$  such that, for all $z$ in $\mathbb{R}$, $\phi_x(z)$   belong to $\{0,1\}$ and such that $\mathrm{d} G_x(z) =\phi_x(z) p^\star(z) \mathrm{d}z$. Therefore, if $g_x(z) = \phi_x(z)  p^\star(z) $, there exists $A$ open subset of $\mathbb{R}$ such that $g_x(z)=0$ almost everywhere in $A$ (there exists $A$ such that $ G_1(A) = 0$, $ G_{2}(A) = 1$,  $ G_1(A^c) = 1$  and $ G_2(A^c) = 0$ ).  In particular $g_x$, or any representative of $g_x$ in $L^1(\mathbb{R})$, can not belong to  $\left\{g_\theta \;,\; \theta \in \Theta\right\}$.
\end{remark}

The purpose of this section is to compare the mixing entropy of the underlying mixture decomposition $\left((\nu^\star_1,\nu^\star_2),\left( g_{\theta^\star_1} ,g_{\theta^\star_2} \right)\right)$ \eqref{entrop:eq:gaussian:pstar} and the entropy of the "merged" version $((1),P^\star)$.

We first provide in Proposition \ref{entrop:prop:gauss:H:simple} below  a nice expression of the minimum of the relative entropy over the Gaussian density family.
 
\begin{proposition}
\label{entrop:prop:gauss:H:simple}
For all $G$ in $\mathcal{M}_1\left( \mathbb{R}\right)$, such that $0<\widehat{\sigma}^2 = \mathbb{E}_G(Z^2) - \mathbb{E}_G(Z)^2<\infty$. Define $\widehat{\mu} = \mathbb{E}_G(Z)$  and  $\widehat\theta = \left( \widehat{\mu},\widehat{\sigma}^2\right)$. Then
\begin{align*}
\min_{\theta = (\mu,\sigma^2)} H\left(G \ ||\ g_\theta \right) &=  H\left(G \ ||\ g_{\widehat\theta} \right)\;,\\
&=  \frac{1}{2} \left( \log\left(\widehat{\sigma}^2 \right) +   \log(2\pi) + 1\right)
\end{align*} 
\end{proposition}
\begin{proof}
The proof is straightforward since, for all $\theta = (\mu,\sigma^2)$, 
\begin{align*}
H\left(G \ ||\ g_\theta \right) &= - \int_\mathbb{R} \log\left[ \frac{1}{\sqrt{2\pi\sigma^2}} \exp \left(-\frac{(z-\mu)^2}{2\sigma^2} \right) \right] G(\rm d z)\\
&= \frac{1}{2} \left[ \log(2\pi) + \log(\sigma^2) +\frac{1}{\sigma^2}\int_\mathbb{R}\left(z-\mu \right)^2 G(\rm d z) \right]\;,
\end{align*}
which is minimized taking $\theta = \widehat{\theta}$.
\end{proof}

Thanks to proposition \ref{entrop:prop:gauss:H:simple}, we can easily show that 
\begin{equation}
\mathbb{H}\left( \left(\nu^\star_1,\nu^\star_2 \right),\left(g_{\theta^\star_1},g_{\theta^\star_2}\right)\right) = H(\nu^\star) + \frac{1}{2} \left[\nu^\star_1 \log({\sigma^\star_1}^2)+\nu^\star_2 \log({\sigma^\star_2}^2) \right] + \frac{1}{2} \left[\log(2\pi) + 1 \right] \;,
\end{equation} 
 and that 
\begin{equation}
\mathbb{H}\left( \left(1\right),\left(P^\star\right)\right) =  \frac{1}{2} \left( \log\left({\sigma^\star}^2 \right) +   \log(2\pi) + 1\right) \;,
\end{equation}
where 
\begin{align*}
{\sigma^\star}^2 &= \mathbb{E}_{P^\star} \left(Z^2\right)- \mathbb{E}_{P^\star} \left(Z\right)^2=  \nu^\star_1{\sigma^\star_1}^2  +  \nu^\star_2{\sigma^\star_2}^2  + \nu^\star_1\nu^\star_2 \left(\mu^\star_1 -\mu^\star_2 \right)^2 \;.
\end{align*}
\begin{proposition}
\label{entrop:prop:identif:gaussian}
 $\mathbb{H}\left( \left(1\right),\left(P^\star\right)\right) >\mathbb{H}\left( \left(\nu^\star_1,\nu^\star_2 \right),\left(g_{\theta^\star_1},g_{\theta^\star_2}\right)\right)$ i.i.f.
 \begin{equation}
 \label{entrop:eq:gaussian:thresh}
 \log\left(\sigma^\star \right) > \nu^\star_1 \log\left(\frac{\sigma^\star_1}{\nu^\star_1 }\right)+\nu^\star_2 \log\left(\frac{\sigma^\star_2}{\nu^\star_2 }\right) \;.
 \end{equation}
Consequently, if Condition \eqref{entrop:eq:gaussian:thresh} is satisfied, then 
\begin{equation}
\label{entrop:eq:gaussian:identif} 
\mathcal{D}^\star_1 \not\subset \mathcal{D}^\star_2
\end{equation} 
 \end{proposition}
 
\begin{remark} 
\label{entrop:rem:prop:gauss}
 \begin{enumerate}
 \item \eqref{entrop:eq:gaussian:identif}  is a notation to assess that for any element $(\left(\nu_1,\nu_2 \right),\left(G_1,G_2 \right)   )$ in  $\mathcal{D}^\star_2$, $\left(\nu_1,\nu_2 \right)$ can not be equal to $(1,0)$ or to $(0,1)$.
 \item  If $\nu^\star_1 =\nu^\star_2 = \frac{1}{2}$ and $\sigma^\star_1=\sigma^\star_2=1$, Condition \eqref{entrop:eq:gaussian:thresh} becomes the following condition on $(\mu^\star_1 ,\mu^\star_2)$:
 \begin{equation}
 \label{entrop:eq:gauss:ident:mu}
 \left|\mu^\star_1 - \mu^\star_2 \right|>2 \sqrt{3}\;.
 \end{equation}
 \item If $\nu^\star_1 =\nu^\star_2 = \frac{1}{2}$, $\mu^\star_1=\mu^\star_2$ and $\sigma^\star_1 = 1$, Condition \eqref{entrop:eq:gaussian:thresh} becomes the following condition on $ \sigma^\star_2$:
 \begin{equation}
\label{entrop:eq:gauss:ident:sigma}
\sigma^\star_2 \in [s_1,s_2]\;,
\end{equation}
where $s_1=4-\sqrt{15}$ and  $s_2=4+\sqrt{15}$.
\end{enumerate}
\end{remark}

\begin{figure}
\begin{center}
\includegraphics[width=0.45\textwidth]{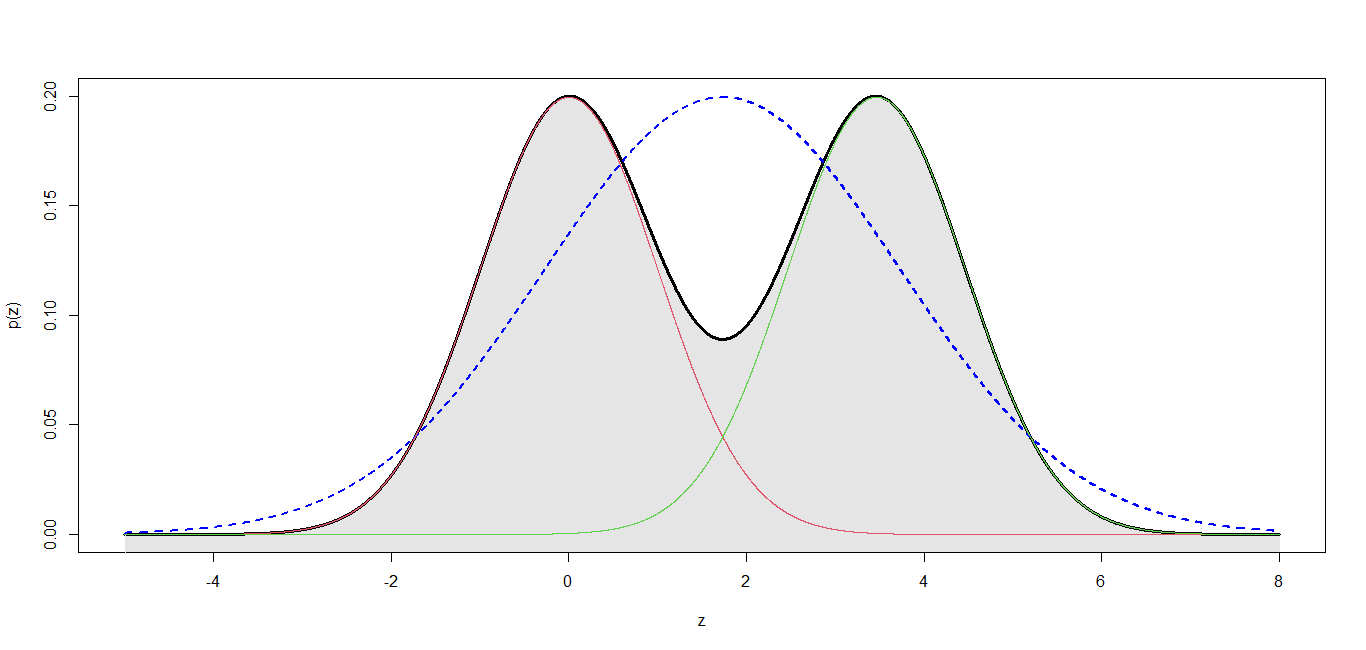}
\includegraphics[width=0.45\textwidth]{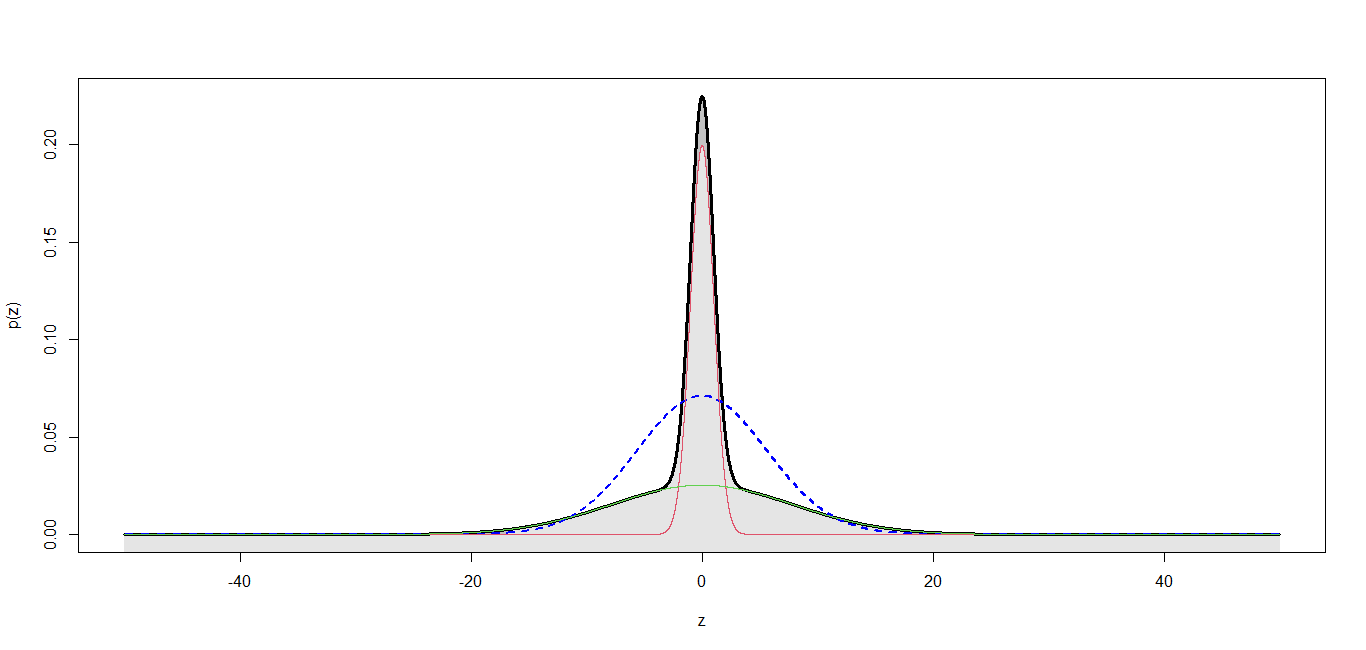}
\caption{Representation of $P^\star$ in the critical cases \eqref{entrop:eq:gauss:ident:mu} (upper graphic with $\mu^\star=2\sqrt{3} $) and   \eqref{entrop:eq:gauss:ident:sigma} (lower graphic with $\sigma^\star = s_2$). In each graphic are represented  $P^\star$ (bold line and filling), the two mixture components $g_{(0,1)}$ and $g_{\theta^\star}$ (thin lines) and $g_{\widehat{\theta}}$, $\widehat{\theta} = \argmin_\theta H\left(P^\star||g_\theta\right)$ (dotted line)}
\label{entrop:fig:identif:illustr}
\end{center}
\end{figure}

Figure \ref{entrop:fig:identif:illustr} represents the cases of equality in  \eqref{entrop:eq:gauss:ident:mu} and  \eqref{entrop:eq:gauss:ident:sigma}.   Despite the relatively large \textit{separation} between the two mixtures, $\mathbb{H}$ is minimum for the \textit{merged} version of this mixture.  
On Figure \ref{entrop:fig:identif:illustr} is  represented the maximum likelihood of $P^\star$ in $\mathcal{C}^\star_1$ : $g_{\widehat{\theta}}$. One can observe that $g_{\widehat{\theta}}$ can be far from the true distribution $P^\star$. However, we should not misinterpret the definition of $\mathcal{D}^\star_2$. Proposition \ref{entrop:prop:identif:gaussian} states that if the component of the mixture are close enough, the  mixture decomposition of $P^\star$ that realizes the minimum mixing entropy  in   $\left(\mathcal{C}^\star_2 \right)$ is $P^\star$ itself. We are not performing here an \textit{estimation} or an approximation of  $P^\star$ by $g_{\widehat{\theta}}$.

\subsection{Mixing entropy in the binary case}
\label{entrop:sec:binary}
We consider here the binary case where $\mathbb{Z}=\{0,1\}$ and where the parameters $\theta = (\mu_0,\mu_1)$  satisfy, for $z$ in $\{0,1\}$, $\mu_z\ge 0$ and $\mu_0+\mu_1 = 1$. We define $g_{\theta} =  \mu_0\delta_0 + \mu_1\delta_1$ and assume that  $P^\star$ belongs to the family $\{g_\theta\;,\;\theta \in \Theta\}$: $$P^\star = \mu^\star_0\delta_0 + \mu^\star_1\delta_1 =g_{\theta^\star}\;.$$ 

The objective of this section is to provide a full description of $\mathcal{D}^\star_r$ for any $r\ge 1$, in this case.

\begin{proposition}
\label{entrop:prop:binary}
For all $r\ge 3$, $\mathcal{D}^\star_r = \mathcal{D}^\star_2$ which means that for all $(\nu , (G_x)_{x=1}^r)$ in $\mathcal{D}^\star_r$, there exits a permutation $\sigma$ of $\{1,\ldots,r\}$ such that, for all $x\ge 3$, $\nu(x)=0$. 

Moreover, for all $((\nu(1),\nu(2)),(G_1,G_2))$ in $\mathcal{D}^\star_2$, either
\begin{itemize}
\item $((\nu(1),\nu(2)),(G_1,G_2))$ belongs to $\mathcal{D}^\star_1$ which means that
\begin{align*}
(\nu(1),\nu(2)) &=(1,0)\mbox{ and }G_1=P^\star \;,\\
\mbox{or }& \\
(\nu(1),\nu(2)) &=(1,0)\mbox{ and }G_2=P^\star \;,
\end{align*}
\item or $\left((\nu(1),\nu(2)),(G_1,G_2)\right)=\left( (\mu^\star_0,\mu^\star_1) , (\delta_0,\delta_1)\right)$
\item or  $\left((\nu(1),\nu(2)),(G_1,G_2)\right)=\left( (\mu^\star_1,\mu^\star_0) , (\delta_1,\delta_0)\right)$.
\end{itemize}

Consequently, if $0<\mu^\star_0,\mu^\star_1<1$, then,
$$\mathcal{D}^\star_1 \varsubsetneq \mathcal{D}^\star_2 = \mathcal{D}^\star_3 = \mathcal{D}^\star_4=\ldots\;,$$
and if  $\mu^\star_0=1$ or $\mu^\star_1=1$ , then, 
$$\mathcal{D}^\star_1 = \mathcal{D}^\star_2 = \mathcal{D}^\star_3 = \mathcal{D}^\star_4=\ldots\;.$$

\end{proposition}
In particular  Proposition \ref{entrop:prop:binary} asserts that the sequence of sets $\left( \mathcal{D}^\star_r\right)_{r\ge 1}$ is constant after either rank $r=1$ or $r=2$ depending on the values of $\mu^\star_0$ and $\mu^\star_1$. 
The proof of Proposition \ref{entrop:prop:binary} is detailed in the Supplementary material \cite{dumont:supp:2022}.

\section{Relative entropic order}
\label{entrop:sec:entropic:order}
The property of constancy after a certain rank of the sequence $\left( \mathcal{D}^\star_r\right)_{r\ge 1}$ induced by Proposition \ref{entrop:prop:binary} in the binary case may be extended to the general case:
\begin{theorem}
\label{entrop:th:rank:def}
Assume A\ref{entrop:hyp:theta:compact}-\ref{entrop:hyp:Z:compact}. The sequence $\left( \mathcal{D}^\star_r\right)_{r\ge 1}$ (resp. $\left( \mathcal{D}^n_r\right)_{r\ge 1}$) is constant after a certain rank $r^\star$ (resp. $r^n$). We call this rank the entropic order of $P^\star$ (resp. of $P^n$) relatively to the family $\{g_\theta,\ \theta\in\Theta\}$.
\end{theorem}
\begin{remark}
\begin{enumerate}
\item Following Remark \ref{entrop:rem:prop:gauss}, this  constancy of $\left( \mathcal{D}^\star_r\right)_{r\ge r^\star}$ (resp. $\left( \mathcal{D}^n_r\right)_{r\ge r^n}$) means that  for all $r>r^\star$ (resp. $r>r^n$),  and all $\left(\widetilde\nu,\left(\widetilde G_x \right)_{x=1}^r\right) $ in $\mathcal{D}^\star_r$ (resp. $\mathcal{D}^n_r$), $ \widetilde\nu$ contains at least $r-r^\star$ (resp. $r-r^n$) zeros.
\item We will also call $r^n$  the empirical relative entropic order. 
\item We could reformulate Theorem \ref{entrop:th:rank:def} as follow: Define  for all $r \ge 1$ and $n \ge 1$, 
\begin{multline}
\label{entrop:eq:rank:star}
\mbox{rank}^\star(r) :=\max \bigg\{ \mbox{card}\{x|\nu(x) >0\} \mbox{ such that } \\
\mbox{ there exists } (G_x)_{x=1}^r \mbox{ satisfying } \left(\nu ,\left(G _x \right)_{x=1}^r\right)\in \mathcal{D}^\star_r\bigg\}
\end{multline} 
\begin{multline}
\label{entrop:eq:rank:n}
\mbox{rank}^n(r) :=\max \bigg\{ \mbox{card}\{x|\nu(x) >0\} \mbox{ such that } \\
\mbox{ there exists } (G_x)_{x=1}^r \mbox{ satisfying } \left(\nu ,\left(G _x \right)_{x=1}^r\right)\in \mathcal{D}^n_r\bigg\}
\end{multline} 
Then there exist $r^\star$ and $r^n \ge 1$ such that for all $r\ge r^\star$, $\mbox{rank}^\star(r)  = r^\star$ and for all $r\ge r^n$, $\mbox{rank}^n(r)  = r^n$ 
\end{enumerate}

\end{remark}
\begin{proof}
The proof is written for $P^\star$. The same arguments hold for  $P^n$.
For all $r\ge 1$, let $\left(\nu^r,\left(G^r_x \right)_{x=1}^r)\right)$ be in $\mathcal{D}^\star_r$. Since the entropic functions are invariant by permutation of $\mathbb{X}$ one can suppose that, for all $r$, $x\mapsto\nu^r(x)$ is non increasing.  The proof of Theorem \ref{entrop:th:rank:def} relies on two basic lemmas:
\begin{lemma}
\label{entrop:lemma:nu1:bounded}
$\nu^r(1)\ge \frac{1}{C^2}$.
\end{lemma}
\begin{proof}
By definition of  $\left(\nu^r,\left(G^r_x \right)_{x=1}^r)\right)$,
$$\sum_{x=1}^r \nu^r(x) \log(\nu^r(x) ) + \sum_{x=1}^r \nu^r(x)\sup_{\theta_x \in\Theta}  \mathbb{E}_{G^r_x} \left(\log(g_{\theta_x}) \right) \ge \sup_{\theta \in \Theta }  \mathbb{E}_{P^\star} \left(\log(g_{\theta}) \right)$$
Then, by A\ref{entrop:hyp:density:bounded} and since $x\mapsto\nu^r(x)$ is non increasing,
$$ \log(\nu^r(1) ) \ge \sum_{x=1}^r \nu^r(x) \log(\nu^r(x) )  \ge -2\log(C)$$

\end{proof}
\begin{lemma}
\label{entrop:lemma:nu:oileffect}
For all $x\in \{1,\ldots,r\}$, $\nu^r(x) = 0$  or $\nu^r(x)\ge \frac{1}{C^2(C^2-1)}$
\end{lemma}
\begin{proof}
Let $x_0$ such that  $\nu^r(x_0) > 0$. We compare the value of $\mathbb{H}\left(\nu^r,\left(G^r_x \right)_{x=1}^r\right)$ with the configuration consisting in merging $x= x_0$ with $x=1$. Define $\left(\widetilde\nu^r,\left(\widetilde G^r_x \right)_{x=1}^r\right)$ by setting $\widetilde G_1^r = \frac{\nu^r(1)G_1^r+\nu^r(x_0)G_{x_0}^r}{\nu^r(1)+\nu^r(x_0)}$, $\widetilde \nu^r(1) = \nu^r(1)+\nu^r(x_0)$ , $\widetilde\nu^r(x_0) = 0$, and $\left(\widetilde\nu^r_x,\widetilde G_x^r\right) =\left( \nu^r_x,  G_x^r\right) $ for $x\notin \{1,x_0\}$. Let, for all $x\le r$, $\theta_x^r$ be a parameter minimizing $\theta \mapsto H\left(G^r_x \ || g_\theta \right)$. 
Then, by a simple manipulation, of $\mathbb{H}_{\theta}$,

\begin{multline*}
 \mathbb{H}_{\theta} \left(\widetilde\nu^r,\left(\widetilde G^r_x \right)_{x=1}^r\right) - \mathbb{H}_{\theta} \left(\nu^r,\left(G^r_x \right)_{x=1}^r\right)  = \\
\nu^r(x_0) \Bigg(\log\left(\nu^r(x_0)  \right) - \frac{\nu^r(1) }{\nu^r(x_0) }\log\left( 1+ \frac{\nu^r(x_0) }{\nu^r(1)}\right) - \log(\nu^r(1)+\nu^r(x_0)) \\
 + \mathbb{E}_{G^r_{x_0}} \left( \log(g_{\theta_{x_0}^r}) \right) - \mathbb{E}_{\widetilde G^r_{1}} \left( \log(g_{\theta_{1}^r}) \right)  \Bigg)
\end{multline*}

And, by A\ref{entrop:hyp:density:bounded}  and Lemma \ref{entrop:lemma:nu1:bounded},

\begin{equation*}
 \mathbb{H}_{\theta} \left(\widetilde\nu^r,\left(\widetilde G^r_x \right)_{x=1}^r\right) - \mathbb{H}_{\theta} \left(\nu^r,\left(G^r_x \right)_{x=1}^r\right)  \le 
\nu^r(x_0)  \log\left(\frac{\nu^r(x_0)C^2}{C^{-2} +\nu^r(x_0)  }  \right)   
\end{equation*}
 
By definition of $\left(\nu^r,\left(G^r_x \right)_{x=1}^r\right) $ and $\theta$, $\mathbb{H}_{\theta} \left(\widetilde\nu^r,\left(\widetilde G^r_x \right)_{x=1}^r\right) - \mathbb{H}_{\theta} \left(\nu^r,\left(G^r_x \right)_{x=1}^r\right) $ can not be negative which implies, since $\nu^r(x_0)>0$, $\frac{\nu^r(x_0)C^2}{C^{-2} +\nu^r(x_0)  }>1$. This concludes the proof.
\end{proof}

We now achieve the proof of Theorem \ref{entrop:th:rank:def}. Assume, by contradiction, that, for all $r^\star\ge 1$, there exists $r> r^\star$ and $x > r^\star$ such that $\nu^r(x) >0$. Then we can build  a  sequence $(u_r)_{r\ge 1}$ growing to infinity, such that for all $r\ge 1$ $\nu^r(u_r)>0$. 
Notice that, necessarily, $\nu^r(u_r)$ converges towards $0$ as $r$ grows to infinity and thus there exists $r$ such that $0<\nu^r(u_r)<\frac{1}{C^2(C^2-1)}$ which contradicts Lemma \ref{entrop:lemma:nu:oileffect}.
Thus,  there exists $r^\star \ge 1$ such that, for all $r>r^\star$ and all $x>r^\star$, $\nu^r(x) = 0$.
\end{proof}

\begin{remark}
\label{entrop:rem:order:bound}
From Lemma \ref{entrop:lemma:nu:oileffect}, $r^\star$ and $r^n$ are necessarily upper bounded by $C^2(C^2-1)$

\end{remark}

\begin{hypH}
\label{entrop:hyp:unicity:rstar}
For any  $\left(\nu^\star,\left(G^\star_x \right)_{x=1}^{r^\star}\right)$ in $\mathcal{D}^\star_{r^\star}$, for all $x$ in $\{1,\ldots,r^\star\}$, $\nu^\star(x) > 0$.
\end{hypH}
 A\ref{entrop:hyp:unicity:rstar} supposes that, if $r^\star>1$, then $\mathcal{D}^\star_{r^\star - 1} \cap \mathcal{D}^\star_{r^\star} = \emptyset$.  
In particular, Assumption A\ref{entrop:hyp:unicity:rstar} excludes the binary case studied in section \ref{entrop:sec:binary} where $\mathbb{H}$ can be minimized both in $\left(\mathcal{C}^\star_{1}\right)$ and in $\left(\mathcal{C}^\star_{2}\right) \setminus \left(\mathcal{C}^\star_{1}\right)$.

\begin{theorem}
\label{entrop:th:rn:converges}
Assume A\ref{entrop:hyp:theta:compact}-\ref{entrop:hyp:Z:compact}.  Asymptotically, almost surely, $r^n$ does not over estimate $r^\star$. Moreover, if A\ref{entrop:hyp:unicity:rstar} holds, then  almost surely, 
$$\lim\limits_{n \to \infty}r^n = r^\star$$
 
\end{theorem}
\begin{proof}
\textit{Asymptotically, $r^n$ does not overestimate $r^\star$: } 
 Using Remark \ref{entrop:rem:order:bound}, let $r_0=  \lceil C^2(C^2-1)\rceil +1$ (where $\lceil \cdot\rceil$ designates the upper whole part), then $r_0\ge r^\star +1$, $r_0\ge r^n+1$ and $\mbox{rank}^\star(r_0)$ and $\mbox{rank}^n(r_0)$ defined by \eqref{entrop:eq:rank:star} and \eqref{entrop:eq:rank:n} satisfy $\mbox{rank}^\star(r_0)= r^\star$ and $\mbox{rank}^n(r_0)= r^n$.

For all $n\ge 1$, let $\left(\nu^n,\left(G^n_x \right)_{x=1}^{r_0}\right)$ in $\mathcal{D}^n_{r_0}$. Assume that $x\mapsto \nu^n(x)$ is decreasing (even if that means permuting the labels $x$).  From Theorem \ref{entrop:th:Dn:converge} and Theorem \ref{entrop:th:rank:def}, if  $\left(\nu^\infty,\left(G^\infty_x \right)_{x=1}^{r_0}\right)$ is a limit of a subsequence   $\left(\nu^{u_n},\left(G^{u_n}_x \right)_{x=1}^{r_0}\right)_{n \ge 1}$, then $\left(\nu^\infty,\left(G^\infty_x \right)_{x=1}^{r_0}\right)$ belongs to $\mathcal{D}^\star$ and, 
$$\mbox{a.s.}\ \lim\limits_{n \to \infty} \nu^{u_n}(r^\star+1)=0 \;,$$
then, by Lemma \ref{entrop:lemma:nu:oileffect},  a.s. $\nu^{u_n}(r^\star+1)=0 $ after a certain rank. This implies that a.s.   there exists $N \ge 1$ such that,  for all $n\ge N$,  $\mbox{rank}^n(r_0) = r^n \le r^\star$. 
 
\textit{Asymptotically, $r^n$ does not underestimate $r^\star$: } 

Let $\left(\nu^n,\left(G^n_x \right)_{x=1}^{r_0}\right)$ in $\mathcal{D}^n_{r_0}$, then any converging subsequence of $\left(\nu^n,\left(G^n_x \right)_{x=1}^{r_0}\right)_{n \ge 1}$ converges in $\mathcal{D}_{r^0}^\star$. However if  A\ref{entrop:hyp:unicity:rstar} holds, then any possible limit  $\left(\nu^\star,\left(G^\star_x \right)_{x=1}^{r_0}\right)$ in $\mathcal{D}_{r^0}^\star$ has exactly $r^\star$ states $x$ such that  $\nu^\star(x) \neq 0$ and, for all $x$, $\lim\limits_{n \to \infty} \nu_n(x) = \nu^\star(x)$. Therefore, $r_n$ can not underestimate $r^\star$ asymptotically.
\end{proof}

\begin{remark} 
Theorem \ref{entrop:th:rn:converges} and Remark \ref{entrop:rem:classif}, insure that the order $r^n$ in the mixing entropy classification method adjusts itself automatically and converges towards the relative order $r^\star$. Unsupervised classification by minimization of the mixing entropy criterion, and, therefore, the classification maximum log-likelihood (see Section \ref{entrop:sec:complete:likelihood} below), are self calibrated methods (adaptive). Unlike classical classification methods such as k-means, the mixing entropy criterion does not encourage to choose the largest number of classes possible.  
\end{remark}
\section{Similarities with the classical mixture models framework }
\label{entrop:sec:simi:mixture:models}
\subsection{Complete likelihood in mixing models}
\label{entrop:sec:complete:likelihood}
In the context of inference in mixing model, if $(Z_1,\ldots,Z_n)$ are observations in $\mathbb{Z}$, the classical MLE, for a given $r\ge 1$ is defined as 

\begin{equation}
\label{entrop:eq:mixing:mle}
\left( \widehat{\nu} , \left(\widehat{\theta_x}\right)_{x=1}^r \right) = \argmax_{(\nu,\theta) \in \mathcal{M}_1\left(\mathbb{X} \right)\times \Theta^r} \log \left(\sum_{x_1=1}^r \ldots \sum_{x_n=1}^r \prod_{k=1}^n \nu(x_k) g_{\theta_{x_k}}\left(Z_k \right) \right) \;.
\end{equation}
Performing the maximization in \eqref{entrop:eq:mixing:mle} is challenging because of the sums appearing inside the $\log$ and methods such as  gradient descent or Expectation-Maximization (EM) algorithm are needed in order to approximate the MLE $\left( \widehat{\nu} , \left(\widehat{\theta_x}\right)_{x=1}^r \right) $.

Now let's focus on the problem of minimization  of the  \textit{complete log-likelihood}, also known as \textit{Classification  log-likelihood \citet{bryant:1991}}, defined as follow: 

\begin{align}
\label{entrop:eq:mixing:complete:likelihood}
\ell_n\left( \left( x_k\right)_{k=1}^n, \nu, \theta \right) :=&  \log\left( \prod_{k=1}^n \nu(x_k) g_{\theta_{x_k}}\left(Z_k \right) \right)\;,\\
=&\sum_{k=1}^n \log\left(  \nu(x_k) \right)  +  \sum_{k=1}^n\log\left( g_{\theta_{x_k}}\left(Z_k \right) \right)\;. \nonumber
\end{align}
 Note that, if $\left( x_k\right)_{k=1}^n$ in $\left\{ 1,\ldots,r\right\}^n$ is set, one can independently  maximize $\ell_n\left( \left( x_k\right)_{k=1}^n, \nu, \theta \right)$ in $\nu$ and $\theta$. In particular, the choice for $\nu$ maximizing $\ell_n\left( \left( x_k\right)_{k=1}^n, \nu, \theta \right)$ is, for all $x$ in $\{1,\ldots,r\}$, $\widehat{\nu}(x) = \frac{1}{n} \sum_{k=1}^n \mathds{1}_x(x_k)$. Therefore, the maximization of the complete log-likelihood  \eqref{entrop:eq:mixing:complete:likelihood} requires the maximization of the function of $\left( x_k\right)_{k=1}^n$:
 
\begin{align}
\ell_{n}\left( \left( x_k\right)_{k=1}^n\right) &:= \sum_{k=1}^n \log\left(  \widehat\nu(x_k) \right)  +  \max_{\theta \in \Theta^r} \sum_{k=1}^n\log\left( g_{\theta_{x_k}}\left(Z_k \right) \right)\;,\nonumber\\
=& n \left[ \sum_{x=1}^r  \widehat\nu(x)\log\left( \widehat\nu(x) \right) + \sum_{x=1}^r \max_{\theta_x \in \Theta}\frac{1}{n}\sum_{k=1}^n \log\left( g_{\theta_{x}}\left(Z_k \right)\right) \mathds{1}_x\left( x_k\right)\right]\;,\nonumber
\end{align}

\begin{remark}
\label{entrop:rem:nz}
Notice that repetitions may occur in the vector  $(Z_1,\ldots,Z_n)$. We  denote by $n_z \ge 1$ the number of $k$ in $\{1,\ldots,n\}$ satisfying $Z_k = z$
\end{remark}
Define  $\phi^{\left( x_k\right)_{k=1}^n}$ in $\Phi_r(\{Z_1,\ldots,Z_n\})$  as:  for all $x$ in $\{1,\ldots,r\}$ and all $z$ in $\{ Z_1,\ldots,Z_n\}$,
$$\phi^{\left( x_k\right)_{k=1}^n}_x(z) =\frac{1}{n_z} \sum_{k\ /\ Z_k = z} \mathds{1}_x\left(x_k\right)\;,$$  
then  
\begin{align}
\ell_{n}\left( \left( x_k\right)_{k=1}^n\right)
=& n \Bigg[ \sum_{x=1}^r  \widehat\nu(x)\log\left( \widehat\nu(x) \right) \nonumber\\
&\quad + \quad \sum_{x=1}^r \max_{\theta_x \in \Theta}\sum_{z\in\{Z_1,\ldots,Z_n\}} \log\left( g_{\theta_{x}}\left(z\right)\right) \left(\frac{n_z}{n}\right)\phi^{\left( x_k\right)_{k=1}^n}_x(z)\Bigg]\;,
\label{entrop:eq:ln:x}
\end{align}

and we recognize, in \eqref{entrop:eq:ln:x}, the mixing entropy of $\phi^{\left( x_k\right)_{k=1}^n}$ : 
\begin{equation}
\label{entrop:eq:ell:phi}
\ell_{n}\left( \left( x_k\right)_{k=1}^n\right)  = - n \cdot \mathbb{H}_{P^n}\left(\phi^{\left( x_k\right)_{k=1}^n}  \right)\;,
\end{equation}
where $$P^n = \sum_{z\in \{Z_1,\ldots,Z_n\}} \left(\frac{n_z}{n} \right) \delta_z\;.$$
Conversely, for every  $\phi$ in $\Phi_r(\{Z_1,\ldots,Z_n\})$, consider $\left(x_1^\phi, \dots, x_n^\phi \right)$ as defined in Remark \ref{entrop:rem:classif}. Then 
\begin{equation}
\label{entrop:eq:phi:ell}
\mathbb{H}_{P^n}\left(\phi \right) = -\frac{1}{n}\ell_{n}\left( \left( x^\phi_k\right)_{k=1}^n\right) \;.
\end{equation}

A consequence of Equations \eqref{entrop:eq:ell:phi} and \eqref{entrop:eq:phi:ell} is that maximizing $\ell_{n}$ in  $\{1,\ldots,r\}^n$ is the same problem as minimizing the entropy $\mathbb{H}_{P^n}$ among $\Phi_r(\{Z_1,\ldots,Z_n\})$. 
 Thus, the minimization of the mixing entropy and the maximization of the classification maximum log-likelihood (CML) of \citet{bryant:1991}  correspond to the exact same problem.

An other consequence is that, thanks to Theorem \ref{entrop:th:rank:def}, there exists $r^n \ge 1$ such that for all $r \ge r^n$, such that the CML satisfies
$$\max_{\left( x_k\right)_{k=1}^n \in \left\{1,\ldots,r \right\}^n} \ell_{n}\left( \left( x_k\right)_{k=1}^n\right) =\max_{\left( x_k\right)_{k=1}^n \in \left\{1,\ldots,r^n \right\}^n} \ell_{n}\left( \left( x_k\right)_{k=1}^n\right)\;.  $$
Moreover, if $\left(Z_1,\ldots,Z_n\right)$ is a an i.i.d. sample of $P^\star$, then Theorem \ref{entrop:th:Dn:converge} implies that  $(r^n)_{n\ge 1}$ is bounded almost surely and if A\ref{entrop:hyp:unicity:rstar} holds, $r^n$ converges almost surely to the entropic order of $P^\star$ relatively to the family   $\left\{g_\theta\;,\;\theta\in\Theta\right\}$.

\subsection{Connection with the EM algorithm}
\label{entrop:sec:EM:algo}
The complete likelihood \eqref{entrop:eq:mixing:complete:likelihood} discussed in Section \ref{entrop:sec:complete:likelihood} appears when implementing the Expectation-Maximization (EM) Algorithm  of \citet{dempster:laird:rubin:1977}. The EM algorithm is an iterative procedure of optimization to approximate the MLE whenever the likelihood takes an integral form which is the case when dealing with mixing models. First introduce the \textit{intermediate quantity}: for all $\nu,\nu'$ in $\mathcal{M}_1(\{1,\ldots,r\})$, and  $\theta,\theta'$ in $\Theta^r$,
\begin{equation*}
Q\left( (\nu,\theta);(\nu' ,\theta')  \right)  := \mathbb{E}_{(\nu' ,\theta') } \left[ \ell_n\left( \left( X_k\right)_{k=1}^n, \nu, \theta \right) | Z_{1},\ldots,Z_n\right]\;,
\end{equation*}
where $\mathbb{E}_{(\nu' ,\theta') } $ is the expectation under the hypothesis that $ (X_k,Z_k)_{k=1}^n$ are i.i.d. with common joint distribution $p(x,\mathrm{d}z) = \nu'(x) g_{\theta'_{x}} (z) \mathrm{d}\lambda z$. Consider an initial parameter value $(\nu^{(0)} ,\theta^{(0)})$.  The EM algorithm consists in repeating the following two step: For all $i \ge 0$, 
\begin{itemize}
\item[\textbf{E-Step}:] Compute $Q\left( (\nu,\theta);(\nu^{(i)} ,\theta^{(i)})  \right) $
\item[\textbf{M-Step}:] Define $(\nu^{(i+1)} ,\theta^{(i+1)})$ as one of the maximizer (provided that it has a sens) of $$(\nu,\theta)\mapsto Q\left( (\nu,\theta);(\nu^{(i)} ,\theta^{(i)})  \right)\;.$$
\end{itemize}
We introduce the shortest notations $ \omega  =(\nu,\theta) $, $x=(x_1,\ldots,x_n) $ in $\{1,\ldots,r\}^n$ and $Z=(Z_1\ldots,Zn)$. 
Then 
\begin{align*}
Q\left( \omega; \omega' \right) = \sum_{x}   p_{\omega'}(x|Z) \log(p_\omega(x,Z)) \;,
\end{align*}
where $p_\omega(x,Z)$ and $  p_{\omega}(x|Z) $  are short notations for 
$p_\omega(x,Z)= \prod_{k=1}^n \nu(x_k)g_{\theta_{}x_k}(Z_k) $  and $p_\omega(x|Z) = p_\omega(x,Z)/\left(\sum_{x'} p_\omega(x',Z) \right)$.  Using these notations, define the  log-likelihood 

$$ \ell(\omega) := \log\left[\sum_x p_\omega(x,Z)  \right]\;,$$

then, the intermediate quantity may be rewritten
\begin{equation}
\label{entrop:eq:intermediate:equality}
Q\left( \omega; \omega' \right)  = \ell(\omega) - H\left( p_{\omega'}(\cdot|Z)\ || \ p_\omega(\cdot|Z) \right) \;.
\end{equation}
 
Relation \eqref{entrop:eq:intermediate:equality} between the intermediate quantity, the objective log-likelihood function $\ell$ and the cross entropy between the conditional distributions provides that, for every $i\ge 0$,

\begin{equation*}
\label{entrop:eq:EM:fonda:equality}
\ell(\omega^{(i+1)}) -\ell(\omega^{(i)})  = Q\left( \omega^{(i+1)};\omega^{(i)}\right) - Q\left( \omega^{(i)};\omega^{(i)}\right) + KL\left( p_{\omega'}(\cdot|Z)\ || \  p_{\omega}(\cdot|Z) \right) 
\end{equation*}
which shows the fundamental inequality of the EM that is

\begin{equation}
\label{entrop:eq:EM:fonda:inequality}
\ell(\omega^{(i+1)}) -\ell(\omega^{(i)})  \ge Q\left( \omega^{(i+1)};\omega^{(i)}\right) - Q\left( \omega^{(i)};\omega^{(i)}\right) \ge 0\;.
\end{equation}
Inequality \eqref{entrop:eq:EM:fonda:inequality} states that the the log-likelihood associated with  the sequence $(\omega^{(i)})_{i\ge 0}$ produced by the EM algorithm ,  $\left(\ell(\omega^{(i)})\right)_{i\ge 0}$ is necessarily non-decreasing.

We now detail the E-Step using the definition \eqref{entrop:eq:mixing:complete:likelihood} of $ \ell_n\left( \left( X_k\right)_{k=1}^n, \nu, \theta \right)$. First  define, for all $x=1,\ldots,r$ and all $z\in \{Z_1,\ldots,Z_n\}$, 

\begin{align}
\phi^{(i)}_x(z) :&= \mathbb{P}_{\omega^{(i)} } \left[ X_1= x \ | \ Z_1=z\right]= \frac{\nu^{(i)}(x) g_{\theta^{(i)}_x}(z)}{\sum_{x' = 1}^r\nu^{(i)}(x') g_{\theta^{(i)}_{x'}}(z)}\;.\label{entrop:eq:phi:EM}
\end{align}
Let $n_z$ be defined as in Remark \ref{entrop:rem:nz}, then

\begin{align}
Q\left( (\nu,\theta);\omega^{(i)}  \right)  =\sum_{x=1}^r \sum_{z \in \{Z_1,\ldots,Z_n\}} n_z  &\phi^{(i)}_x(z) \log (\nu(x)) \label{entrop:eq:Qphi}\\
+  \sum_{x=1}^r \sum_{z \in \{Z_1,\ldots,Z_n\}}n_z  &\phi^{(i)}_x(z) \log (g_{\theta_x}\left(z \right))\;. \nonumber
\end{align}

Using the definition \eqref{entrop:eq:Gx},  let $G_x^{\phi^{(i)}}$ be the probability distribution on $\{Z_1,\ldots,Z_n\}$ define by 

\begin{equation}
\label{entrop:eq:G:phi}
G_x^{\phi^{(i)}} =\frac{  \sum_{z \in \{Z_1,\ldots,Z_n\}} n_z \phi^{(i)}_x(z) \delta_z}{\sum_{z \in \{Z_1,\ldots,Z_n\}} n_z \phi^{(i)}_x(z)}\;,
\end{equation}
where $\delta_z$ is the notation for the Dirac distribution on the singleton $\{z\}$. Then 
\begin{equation}
\label{entrop:Q:H:corresp}
Q\left( (\nu,\theta);\omega^{(i)}  \right)  = -n \mathbb{H}_\theta \left(\nu,\left(G_x^{\phi^{(i)}}\right)_{x=1}^r \right)\;.
\end{equation} 
Moreover, if $P^n = \sum_{z\in \{Z_1,\ldots,Z_n\}}\left( \frac{n_z}{n} \right) \delta_z$, then the EM algorithm also provides a sequence of elements of $\Phi_r(\{Z_1,\ldots,Z_n\})$: $ \left(\phi^{(i)}\right)_{i\ge 0}$ such that, their corresponding mixing entropy satisfies, for all $i\ge 0$, 
\begin{equation}
\label{entrop:eq:H:phi:EM}
\mathbb{H}_{P^n} \left( \phi^{(i)}\right) = -\frac{1}{n}Q\left( \omega^{(i+1)};\omega^{(i)}  \right)\;.
\end{equation} 

Note that, while the sequence of log-likelihood $\left(\ell(\omega^{(i)})\right)_{i\ge 0}$ produced by the EM algorithm  is necessarily non-decreasing, we have no guaranty that the sequence $\left(\mathbb{H}_{P^n} \left( \phi^{(i)}\right)\right)_{i\ge 0} $ is non-increasing. A slight modification of the EM algorithm proposed in  \citet{celeux:1992} will allow us to construct a non-increasing mixing entropy sequence.  

Before presenting this algorithm, we  introduce the following notation: for all $z\in \{Z_1,\ldots,Z_n\}$ and all $x$ in $\{1,\ldots,r\}$, denote

\begin{align}
[\phi]_x^{(i)}(z) &= 1 \mbox{ if } \phi_x^{(i)}(z) = \max_{x'}\phi_{x'}^{(i)}(z) \label{entrop:eq:phicrochet}\\
&= 0 \mbox{ otherwise.} \nonumber
\end{align} 
The computation of $[\phi]_x^{(i)}$ is equivalent to the computation of the maximum a posteriori  (MAP) estimator in the mixture model defined by $\omega^{(i)}$.
\begin{proposition}\label{entrop:phi:class:better} 
For all $i\ge 0$,
\begin{equation*}
\mathbb{H}_{P^n} \left([\phi]^{(i+1)}\right)  \le \mathbb{H}_{P^n} \left( \phi^{(i)}\right) 
\end{equation*}
\end{proposition}
\begin{proof}
For all $\phi$ in $\Phi_r$, 
\begin{align*}
\mathbb{H}_{\theta^{(i)}} \left( \nu^{(i+1)}, \left(G_x^{\phi} \right)_{x=1}^r \right) & = - \Bigg[ \sum_{z\in\{Z_1,\ldots,Z_n\}}n_z \sum_{x=1}^r  \phi_x(z) \log \left( \nu^{(i+1)}(x) g_{\theta^{(i+1)}_x}(z) \right)\Bigg]\;.
\end{align*} 
Note that, for all $z$ in $\{Z_1,\ldots,Z_n\}$, $ \sum_{x=1}^r  \phi_x(z) \log \left( \nu^{(i+1)}(x) g_{\theta^{(i+1)}_x}(z) \right) $ is maximized under the constraints  $\sum_{x=1}^r \phi_x(z) = 1$ and $\phi_x(z)\ge 0$   when $\phi$ satisfies: $ \phi_x(z) =1$ if $x$ maximizes $x\mapsto \nu^{(i+1)}(x) g_{\theta^{(i+1)}_x}(z)$, and $ \phi_x(z) = 0$ otherwise, which is when  $\phi =[\phi]^{(i+1)} $.  Then
\begin{multline*}
\mathbb{H}_{P^n} \left([\phi]^{(i+1)}\right)  \le \mathbb{H}_{\theta^{(i+1)}} \left( \nu^{(i+1)}, \left(G_x^{[\phi]^{(i+1)}} \right)_{x=1}^r \right)\\
 \le \mathbb{H}_{\theta^{(i+1)}} \left( \nu^{(i+1)}, \left(G_x^{\phi^{(i)}} \right)_{x=1}^r \right) = \mathbb{H}_{P^n} \left(\phi^{(i)}\right)  
\end{multline*} 
where the last equality comes from \eqref{entrop:Q:H:corresp} and \eqref{entrop:eq:H:phi:EM}.

\end{proof}

Now consider the Classification EM algorithm (CEM), introduced by \citet{celeux:1992}, and rewritten here using the entropy notations (thanks to Equation \eqref{entrop:Q:H:corresp}). The CEM Algorithm is described by Algorithm \ref{entrop:algo:CEM}.
 
\begin{algorithm}
\caption{CEM algorithm of \citet{celeux:1992}}
\label{entrop:algo:CEM}
\begin{algorithmic} 
\REQUIRE  $(\nu^{(0)} ,\theta^{(0)})$ an initial parameter value.
\STATE Repeat the following three steps until convergence
\STATE \textbf{E-Step}: Compute  $\phi^{(i)}$ following \eqref{entrop:eq:phi:EM}.
\STATE \textbf{C-Step}: Compute the MAP estimator $[\phi]^{(i)}$ using \eqref{entrop:eq:phicrochet}.
\STATE \textbf{M-Step}: Define $(\nu^{(i+1)} ,\theta^{(i+1)})$ as one of the minimizer (provided that it has a sens) of $$(\nu,\theta)\mapsto \mathbb{H}_\theta \left(\nu,\left(G_x^{[\phi]^{(i)}}\right)_{x=1}^r \right)\;.$$
\end{algorithmic}
\end{algorithm} 

Proposition \ref{entrop:prop:decreasing:entrop:CEM} is a straightforward generalization of Proposition 2 of \citet{celeux:1992}. Its proof relies on the adaptation of Proposition \ref{entrop:phi:class:better} to the sequence $\left(\phi^{(i)}\right)_{i \ge 0}$ re-defined in Algorithm \ref{entrop:algo:CEM}.
\begin{proposition}[Proposition 2 of \citet{celeux:1992}]
\label{entrop:prop:decreasing:entrop:CEM}
 if A\ref{entrop:hyp:theta:compact}-\ref{entrop:hyp:Z:compact} are satisfied, the sequence $$\left(\mathbb{H}_{P^n} \left([\phi]^{(i)}\right) \right)_{i \ge 0} \;,$$ produced by Algorithm \ref{entrop:algo:CEM}, is non-increasing and converges to a stationary value. Moreover, the sequences $\left([\phi]^{(i)}\right)_{i \ge 0}$ and $\left(\nu^{(i)},\theta^{(i)}\right)_{i \ge 0}$ remain constant after a certain $i_0$.
\end{proposition}

\section{Practical implementation of  $\mathcal{D}^n_r$} 
\label{entrop:sec:practical:implement}
The empirical version of \eqref{entrop:eq:parallel:phi:nuG} is

\begin{equation*} 
\min_{ \Phi_r(\{Z_1,\ldots,Z_n\})}\mathbb{H}_{P^n}  = \min_{ \left(\mathcal{C}^n_r \right) } \mathbb{H} \;.
\end{equation*} 
We can thus focus  on the practical computation of $\Phi^n_r$ defined by \eqref{entrop:eq:phi:n:r}. From Theorem \ref{entrop:th:phistar}, for all $\phi = (\phi_1,\ldots,\phi_r)$ in  $\Phi^n_r$, for all $x$ in $\{1,\ldots,r\}$ and all $z$ in $\{Z_1,\ldots,Z_n\}$,  $\phi_x(z) = 0$ or  $\phi_x(z) = 1$. 
Thus, the number of potential functions $\phi$ in $\Phi^n_r$ corresponds to the number of possible classifications $(X_k)_{k=1,\ldots,n}$ in $\{1,\ldots,r\}^n$, which grows exponentially with $n$. Thus, the exact of minimization of $\phi \mapsto \mathbb{H}_{P^n}(\phi)$ is a NP-hard problem.

Proposition \ref{entrop:prop:decreasing:entrop:CEM} shows that Algorithm \ref{entrop:algo:CEM} in Section \ref{entrop:sec:EM:algo} provides a non-decreasing sequence of $\left(\mathbb{H}_{P^n} \left([\phi]^{(i)}\right) \right)_{i \ge 0} $. However we showed that the minimum mixing entropy does not decrease when  $r$ exceeds the relative entropic order $r^n$ and Algorithm \ref{entrop:algo:CEM}, that is defined for a given order $r$, does not take that property into account. Moreover, when executing Algorithm \ref{entrop:algo:CEM} with large values of $r$, the  C-step provokes  an  extinction of some classes after the first loops of the algorithm meaning that, for small values $i$ and  some $x$ in $\{1,\ldots,r\}$, $[\phi]^{(i)}_x(z) = 0$, for all $z$ in $\{Z_1,\ldots,Z_n\}$, without giving the opportunity to the EM algorithm to \textit{reorganize} the data.  Algorithm \ref{entrop:algo:phi:build} is an alternative that browses a larger set of $\phi$'s. It runs the sequences $(\phi^{(i)})_{i\ge 0}$ produced by the classical EM algorithm initiated with $N_{init}$ random values $\phi^{(0)}$ rather that initiating with initial parameters $(\nu^{(0)},\theta^{(0)})$ like EM and CEM algorithm do. The considered values for $r$ grow until no improvement is made (after $\mbox{stop}_{r} \ge 1$ increasing values of $r$ without any improvement of the mixing entropy). Finally, exploiting  Proposition \ref{entrop:phi:class:better}, Algorithm \ref{entrop:algo:phi:build} runs the classifier $[\phi]$  in parallel  at each step and tests if the mixing-entropy  decreases or not.
\begin{algorithm}
\caption{Pseudo-code for the construction of $\widehat{\phi}$}
\label{entrop:algo:phi:build}
\begin{algorithmic} 
\REQUIRE $N_{init}\ge 1$ $\mbox{stop}_{em} \geq 1$, $\mbox{stop}_{r} \geq 1$, 
\STATE $\mbox{ind}_{em} = 1$, $\mbox{ind}_r= 1$,
\STATE $\widehat{H} = \infty$, 
\STATE $r=1$, $\mbox{ind}_r = 0$,
\WHILE{$\mbox{ind}_r<\mbox{stop}_{r}$}
	\STATE  $\mbox{ind}_{r} 	= \mbox{ind}_{r} + 1$,
	\FOR{$\mbox{init}$ in $\{1,\ldots,N_{init}\}$}
		\STATE randomly initialize $\phi^{(0)}_x(z)$, $x=1,\ldots,r$, $z \in \{Z_1,\ldots,Z_n\}$  in $ \Phi_r(\{Z_1,\ldots,Z_n\})$,
		\STATE $i=0$, $\mbox{ind}_{em} = 0$,
		\WHILE{$\mbox{ind}_{em}<\mbox{stop}_{em}$}
			\STATE $\mbox{ind}_{em} = \mbox{ind}_{em}+1$,
			\STATE define $\omega_{i+1} = \argmin_{\omega =(\nu,\theta)}\mathbb{H}_\theta \left(\nu,\left(G_x^{\phi^{(i)}}\right)_{x=1}^r \right)$ using  \eqref{entrop:eq:G:phi},
			\STATE define $\phi^{(i+1)} $ using \eqref{entrop:eq:phi:EM} and   $[\phi]^{(i+1)}$ using \eqref{entrop:eq:phicrochet},
			\STATE calculate $[H]=\mathbb{H}_{P^n} \left([\phi]^{(i+1) }\right)$,
			\IF{$[H]<\widehat{H}$}
				\STATE Update:
				\begin{equation*}
				\widehat{H}  = [H] \;,\ 
				\widehat{\phi} = [\phi]^{(i+1)} \;.
				\end{equation*}
				\STATE Reset: $\mbox{ind}_{em} 	= 0$ and $\mbox{ind}_{r} 	= 0$.
				
			\ENDIF 
			\STATE $i=i+1$,
		\ENDWHILE
	\ENDFOR
	\STATE $r = r+1$.
\ENDWHILE 
\end{algorithmic}
\end{algorithm} 
\section{Illustration with synthetic data}
\label{entrop:sec:numerical} 
In this section we will execute Algorithm \ref{entrop:algo:phi:build} on synthetic data. We do not intend to provide an exhaustive analysis of the performance of this method. Our purpose is to illustrate the results discussed through out the paper. 

We choose for the underlying distribution $P^{\star}$  of the synthetic data a Gaussian mixture distributions with various order and parameters. We also perform our classification relatively to two classes of densities: the classical Gaussian densities: 

\begin{align}
\label{entrop:eq:gauss:density}
g_{\theta} (z) = \frac{1}{\sqrt{2\pi\sigma^2} }\exp \left( -\frac{(z-\mu)^2}{2\sigma^2}\right) \;, \; \theta =(\mu,\sigma) \in \mathbb{R}\times ]0,+\infty[
\end{align}
and the \textit{bi-sided, asymmetrical exponential} densities defined as:
\begin{equation}
\label{entrop:eq:biexp:density}
g_{\theta} (z) = p  \lambda_R \exp \left( - \lambda_R \left(Z-A_R \right)\right)  \mathds{1}_{z\ge A_R} 
+ (1-p)  \lambda_L \exp \left( - \lambda_L \left(A_L - Z \right)\right) \mathds{1}_{z\le A_L} 
\end{equation}
where  $ p \in [\alpha,1-\alpha]$, $ -\infty<A_L\le A_R <\infty $ and $\lambda_L,\lambda_R>0$. An illustration of such a density is provided in Figure \ref{entrop:fig:expBiSideded:illustr}. 

\begin{figure}
\begin{center}
\includegraphics[width=150pt]{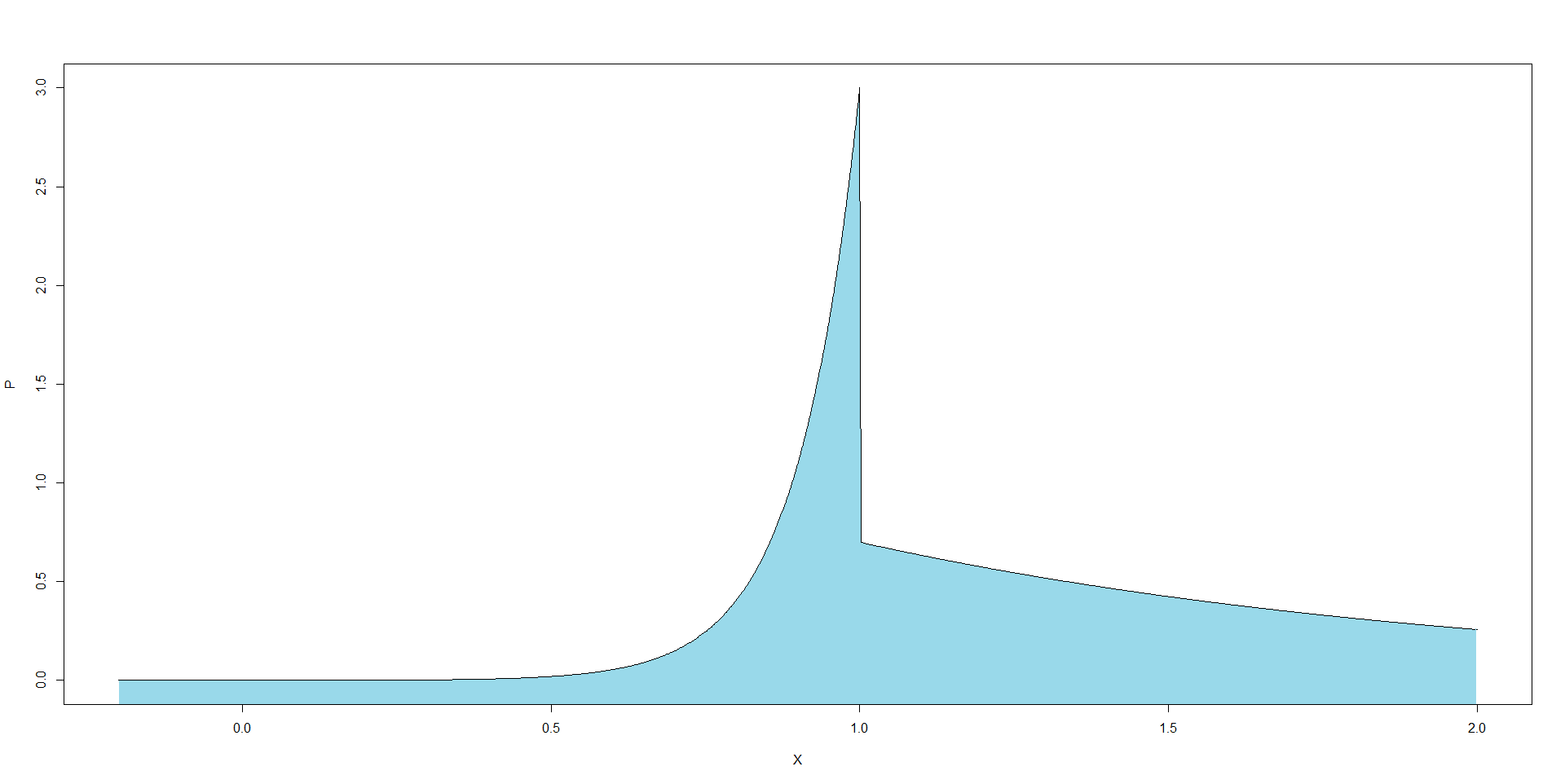}
\includegraphics[width=150pt]{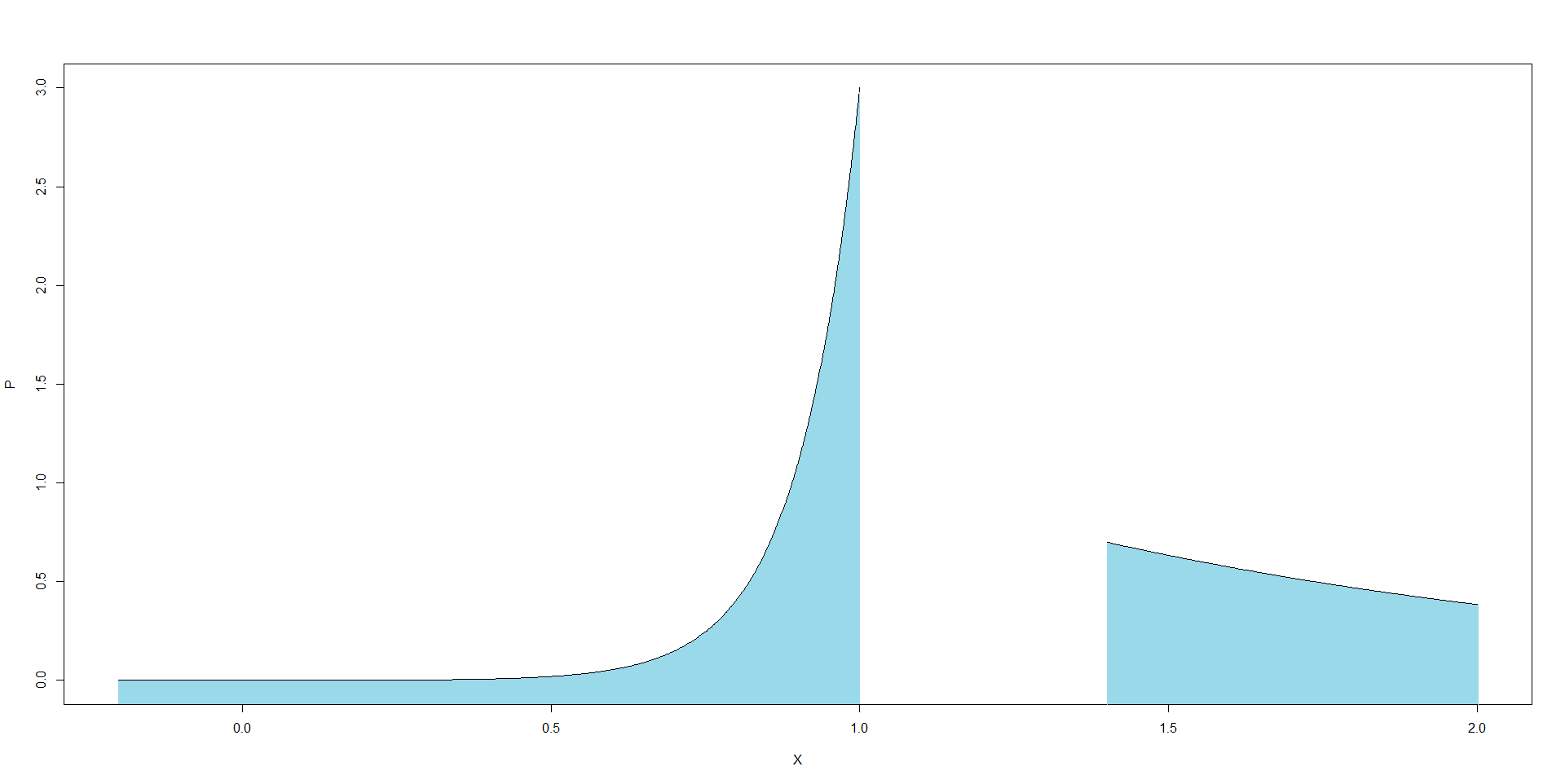}
\caption{Representation of a bi-sided, asymmetrical exponential density. For both graphics, $p = 0.7$,  $A_L = 1$, $\lambda_R = 1$, $\lambda_L = 10$. On the graphic on the left hand side $A_R = 1$, and on the right hand side $A_R = 1.2$.}
\label{entrop:fig:expBiSideded:illustr}
\end{center}
\end{figure}
\begin{remark}
\begin{enumerate}
\item Following Remark \ref{entrop:rem:gaussian:setting}, neither the Gaussian family nor the bi-sided exponential familly satisfies Assumptions A\ref{entrop:hyp:density:bounded} and $\mathbb{R}$ is not a compact metric space. The minimization of $\mathbb{H}\left( \nu,\left( G_x\right)_{x=1}^r\right)$, on $\left(\mathcal{C}^n_r \right)$, with $r\ge 2$, gives $\min_{\left(\mathcal{C}^n_r \right)} \mathbb{H} = - \infty$. Indeed, if we concentrate $G_1$ on only one value of $z$ by, for instance, taking $\nu(1) = \frac{1}{n}$ and $G_1 = \delta_{Z_1}$. Then considering $\mu_1 = Z_1$ and $\sigma_1 \to 0$ in the Gaussian setting or $A_R =Z_1$ and $\lambda_R \to +\infty$ in the  bi-sided, asymmetrical exponential setting  gives  $\mathbb{H} \left( \nu,\left( G_x\right)_{x=1}^r\right) = - \infty$. It is conceivable to restrict $\left\{g_\theta\ , \ \theta \in \Theta\right\}$ by bounding the parameter sets in order to avoid such behavior. However, when running Algorithm \ref{entrop:algo:phi:build}, such concentration phenomenon do not appear before Algorithm \ref{entrop:algo:phi:build} stops, and no such restrictions were needed to obtain the practical results in this section.
\item We could consider any other class of density. For instance, on the same model, we could define bi-sided, asymmetrical Gaussian densities. The choice of exponential is here arbitrary. 
\item The restriction $ p \in [\alpha,1-\alpha]$ (we choose $\alpha$ very small in practice)  is made to avoid the phenomenon  discussed in Section \ref{entrop:sec:binary}, that is $\mathcal{D}^\star_{r^\star-1} \subset\mathcal{D}^\star_{r^\star}$: if $$\theta^\star = ( p^\star,A^\star_L, A_R^\star,\lambda^\star_L,\lambda^\star_R)$$
if $\theta^\star_1 = ( 1 ,-, A_R^\star,-,\lambda^\star_R)$ and $\theta^\star_2 = ( 0,A^\star_L, -,\lambda^\star_L,-)$
($p=1$ in $\theta^\star_1$ and $p=0$ in $\theta^\star_2$ ), if
 $P^\star = g_{\theta^\star} (z) \mathrm{d}z$, then $\mathbb{H} \left( (1),\left(P^\star\right)\right) = \mathbb{H} \left( (p^\star, 1-p^\star ), \left(g_{\theta^\star_1},g_{\theta^\star_2}\right)\right)$,
\end{enumerate}
and the two mixture decompositions with different orders provide the same mixing entropy  and  Assumption \ref{entrop:hyp:unicity:rstar} can not be satisfied.
\end{remark}

First start with the application of Algorithm \ref{entrop:algo:phi:build}, when the synthetic data is generated as an i.i.d. sample of size $n=10000$ on $\mathbb{R}$ of $\mathrm{d}P^\star(z) =p^\star(z)\mathrm{d}z $ with 
\begin{equation}
\label{entrop:eq:Pstar:r2}
p^\star(z) = \frac{1}{2} g_{0,1}(z) +\frac{1}{2} g_{\mu^\star,1}(z) \;,
\end{equation}
where $\mu^\star>0$ and $g_\theta$ is given by \eqref{entrop:eq:gauss:density}. 
 
\begin{table}
\begin{center}
\begin{tabular}{|c|c|c|}
\hline 
$\mu^\star$ & Gaussian Family \eqref{entrop:eq:gauss:density} & Exponential family \eqref{entrop:eq:biexp:density} - $\alpha=0.005$ \\
\hline 
$0.70\cdot(2 \sqrt{3})$ & 1 & 1 \\ 
$0.75\cdot(2 \sqrt{3})$ & 1 & 2 \\ 
$0.9\cdot(2 \sqrt{3})$ & 1 & 2 \\ 
$2 \sqrt{3}$& 1 & 2 \\ 
$1.1\cdot(2 \sqrt{3})$ & 2 & 2 \\
\hline 
\end{tabular}
\caption{Values of $r^n$ for different values of $\mu^\star$ in the case  \eqref{entrop:eq:Pstar:r2}}
\label{entrop:table:case:r2}
\end{center}
\end{table} 

\begin{figure}
\begin{center}
\includegraphics[width=110pt]{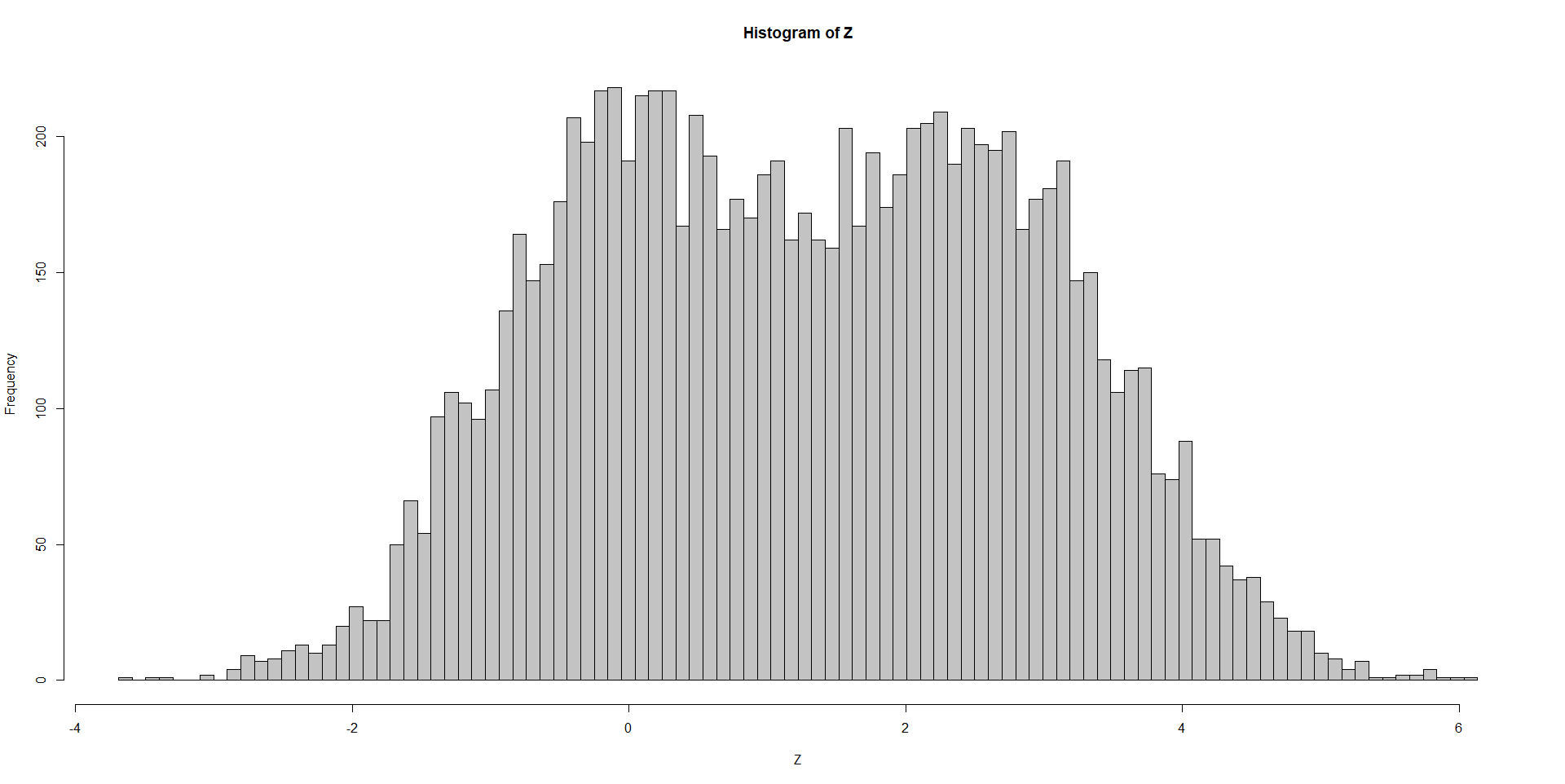} 
\includegraphics[width=110pt]{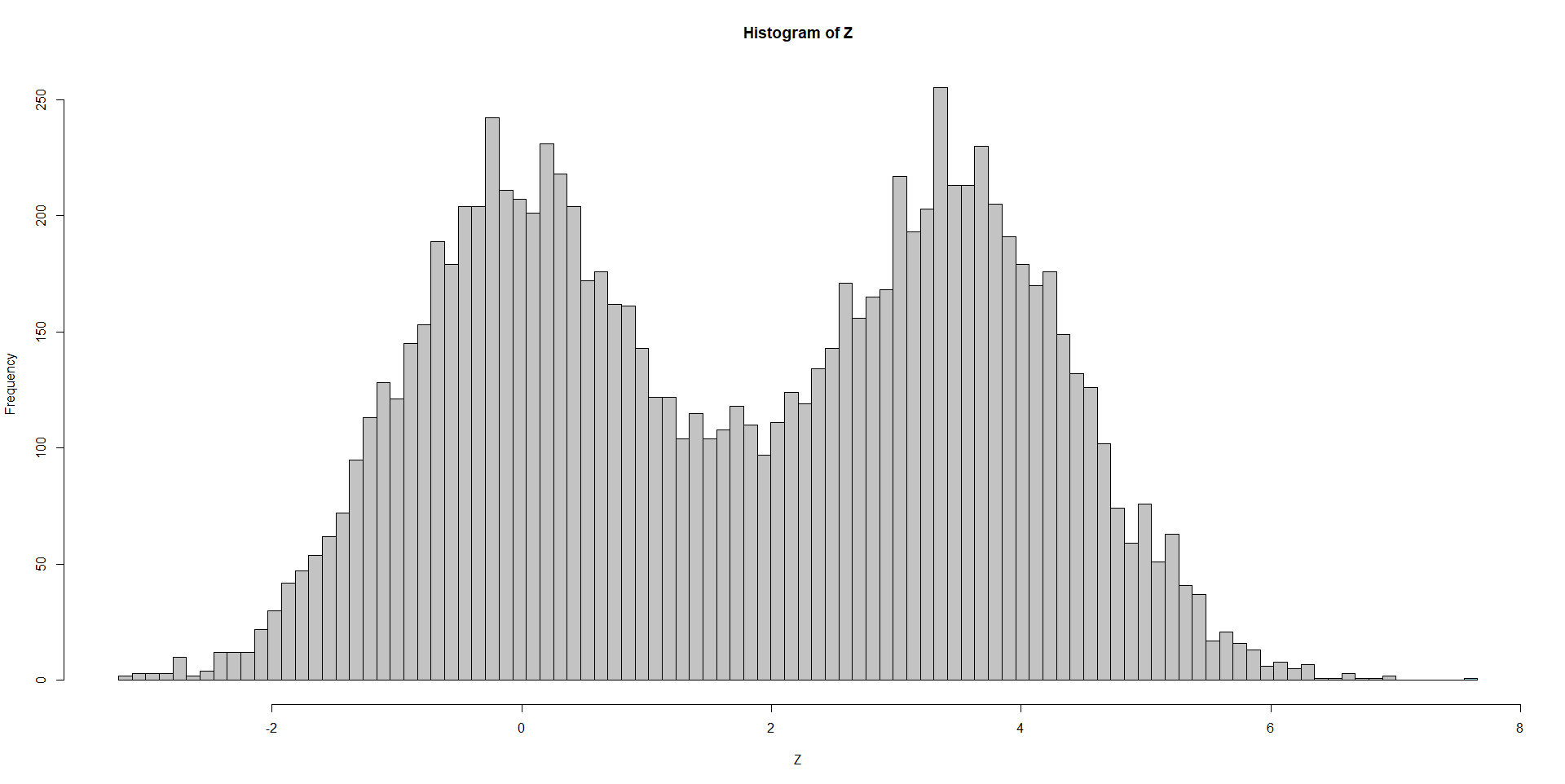} 
\includegraphics[width=110pt]{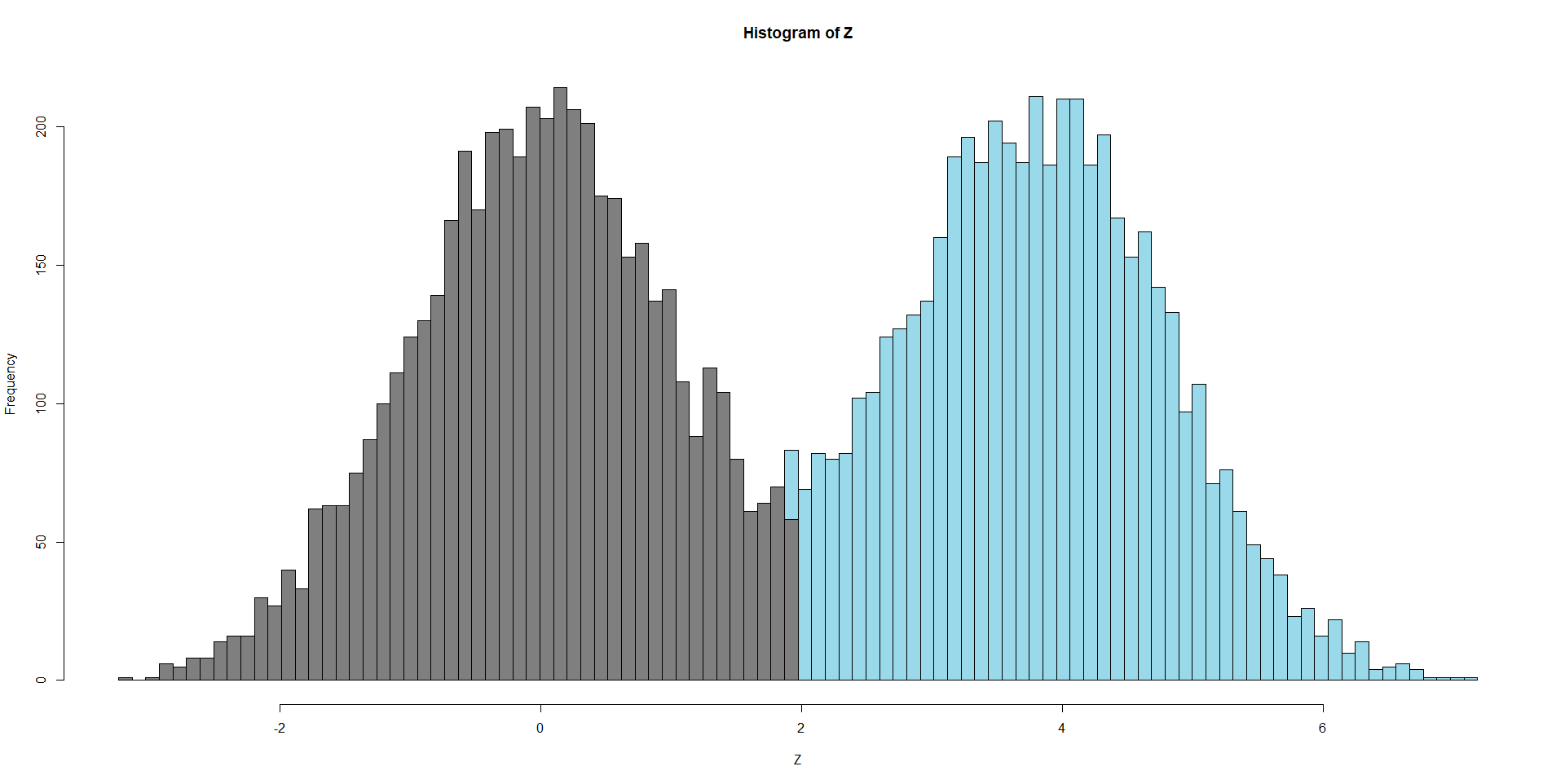} \\
\includegraphics[width=110pt]{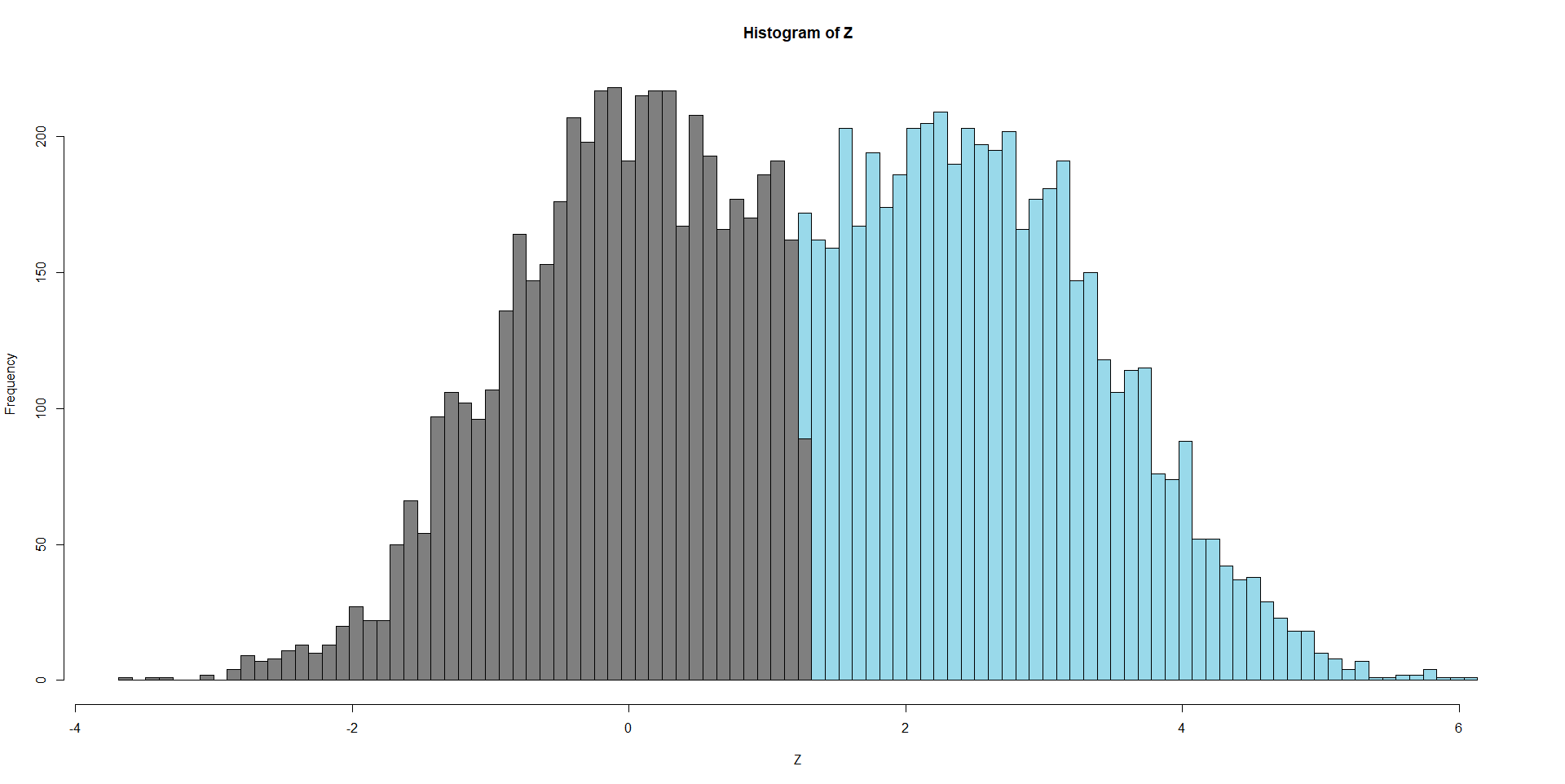} 
\includegraphics[width=110pt]{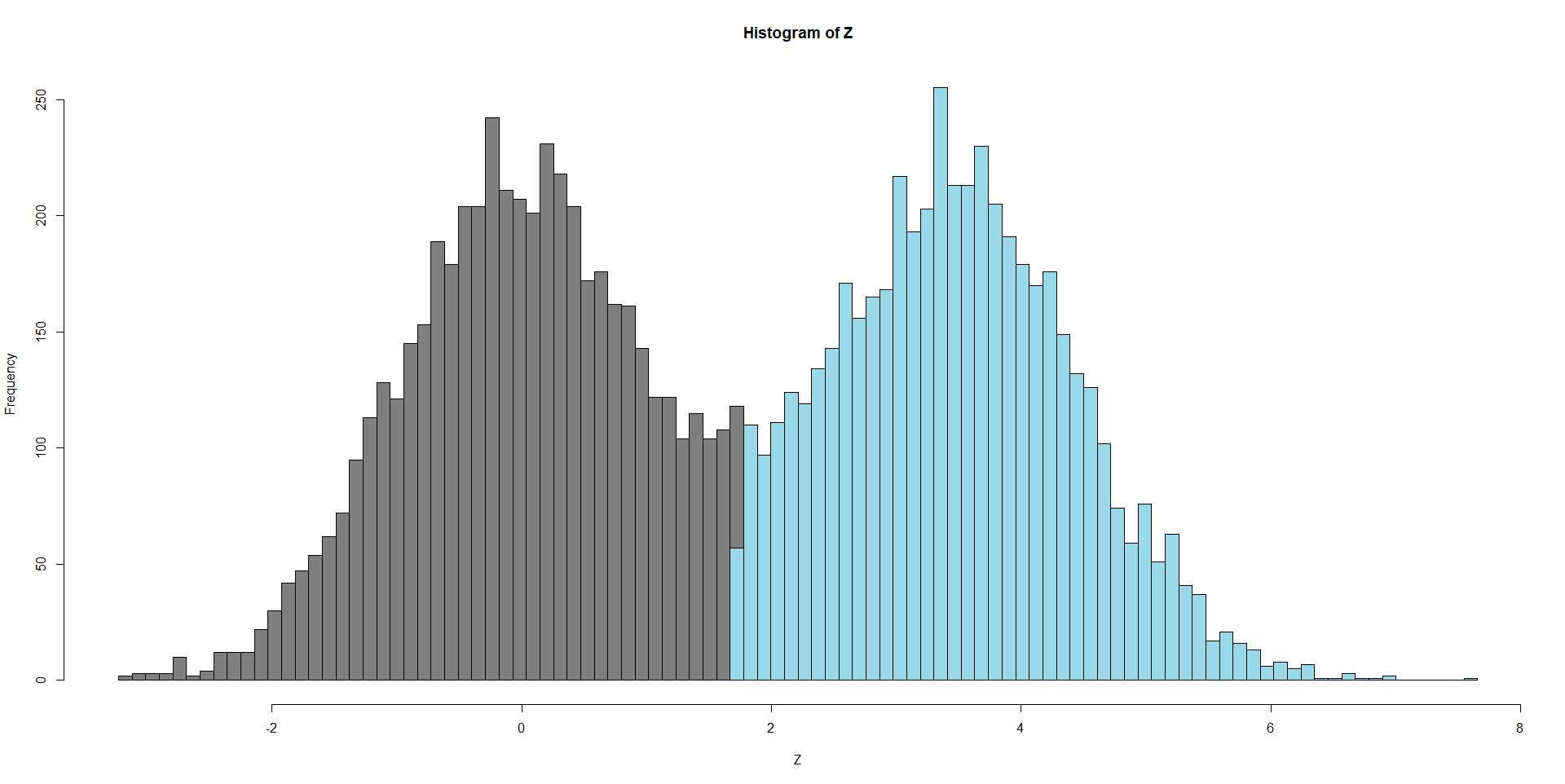} 
\includegraphics[width=110pt]{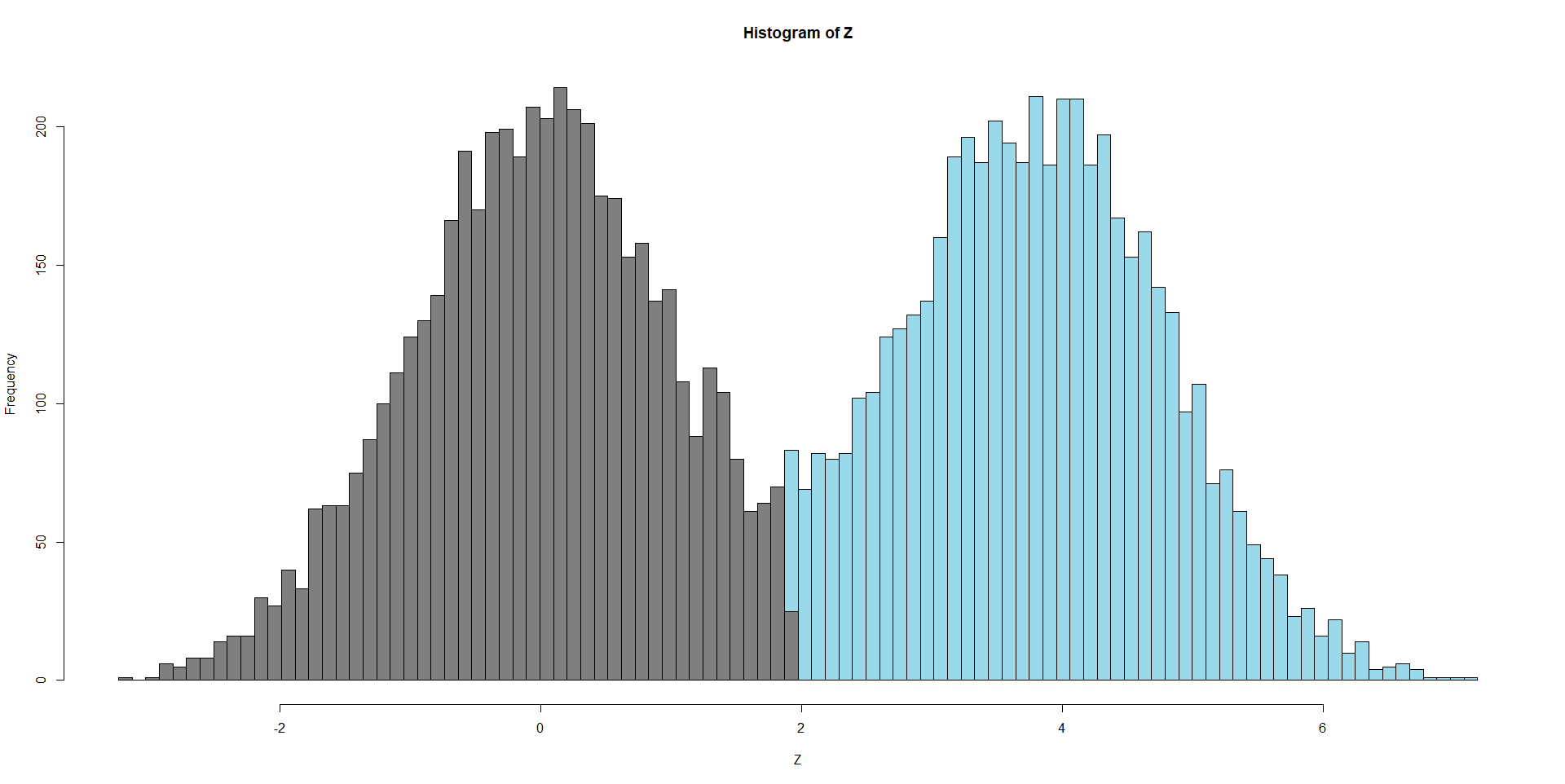} 
\caption{
Histograms of the data regrouped by class  (corresponding to Table \ref{entrop:table:case:r2}) with, $\mu^\star = 0.75*(2 \sqrt{3})$ (Left),  $\mu^\star = 2 \sqrt{3}$ (Middle), $\mu^\star = 1.1*(2 \sqrt{3})$ (Right), relatively to the Gaussian family  \eqref{entrop:eq:biexp:density} (Top) and the exponential family \eqref{entrop:eq:biexp:density} (Bottom) }
\label{entrop:fig:case:r2:illustr}
\end{center}
\end{figure}

Table \ref{entrop:table:case:r2} and Figure \ref{entrop:fig:case:r2:illustr} represent the results obtained for different values of the only free parameter in this case : $\mu^\star$. We observe that, when dealing with the Gaussian family, the threshold between the cases $r^n= 1$ and $r^n =2$ occurs somewhere near the theoretical threshold obtained in \eqref{entrop:eq:gauss:ident:mu}: $\mu^\star = 2 \sqrt{3}$, whereas the threshold is smaller when using the bi-sided asymmetrical  exponential family. This may be explained by the richness of the bi-sided asymmetrical  exponential family compared with the classical, symmetrical Gaussian family.

We confirm this observation with  the second application of Algorithm \ref{entrop:algo:phi:build} where we assume that $p^\star(z)$ is given by 

\begin{equation}
\label{entrop:eq:Pstar:r7}
p^\star(z) = \frac{1}{7} \sum_{x=1}^7  g_{(x-1)\cdot\mu^\star,1}(z) = \frac{1}{7} \left( g_{0,1 }(z) +  g_{\mu^\star,1 }(z) +\ldots + g_{6\mu^\star ,1}(z)  \right)\;,
\end{equation}
where $\mu^\star>0$ and $g_\theta$ is given by \eqref{entrop:eq:gauss:density}. 
\begin{table}
\begin{center}
\begin{tabular}{|c|c|c|}
\hline 
$\mu^\star$ & Gaussian Family \eqref{entrop:eq:gauss:density} & Exponential family \eqref{entrop:eq:biexp:density} - $\alpha=0.005$\\
\hline 
$0.50\cdot(2 \sqrt{3})$ & 1 & 1 \\ 
$0.60\cdot(2 \sqrt{3})$ & 1 & 3 \\ 
$0.70\cdot(2 \sqrt{3})$ & 1 & 7 \\   
$2 \sqrt{3}$& 7 & 7 \\  
\hline 
\end{tabular}
\caption{Values of $r^n$  for different values of $\mu^\star$ in the case  \eqref{entrop:eq:Pstar:r7}}
\label{entrop:table:case:r7}
\end{center}
\end{table} 

\begin{figure}
\begin{center}
\includegraphics[width=140pt]{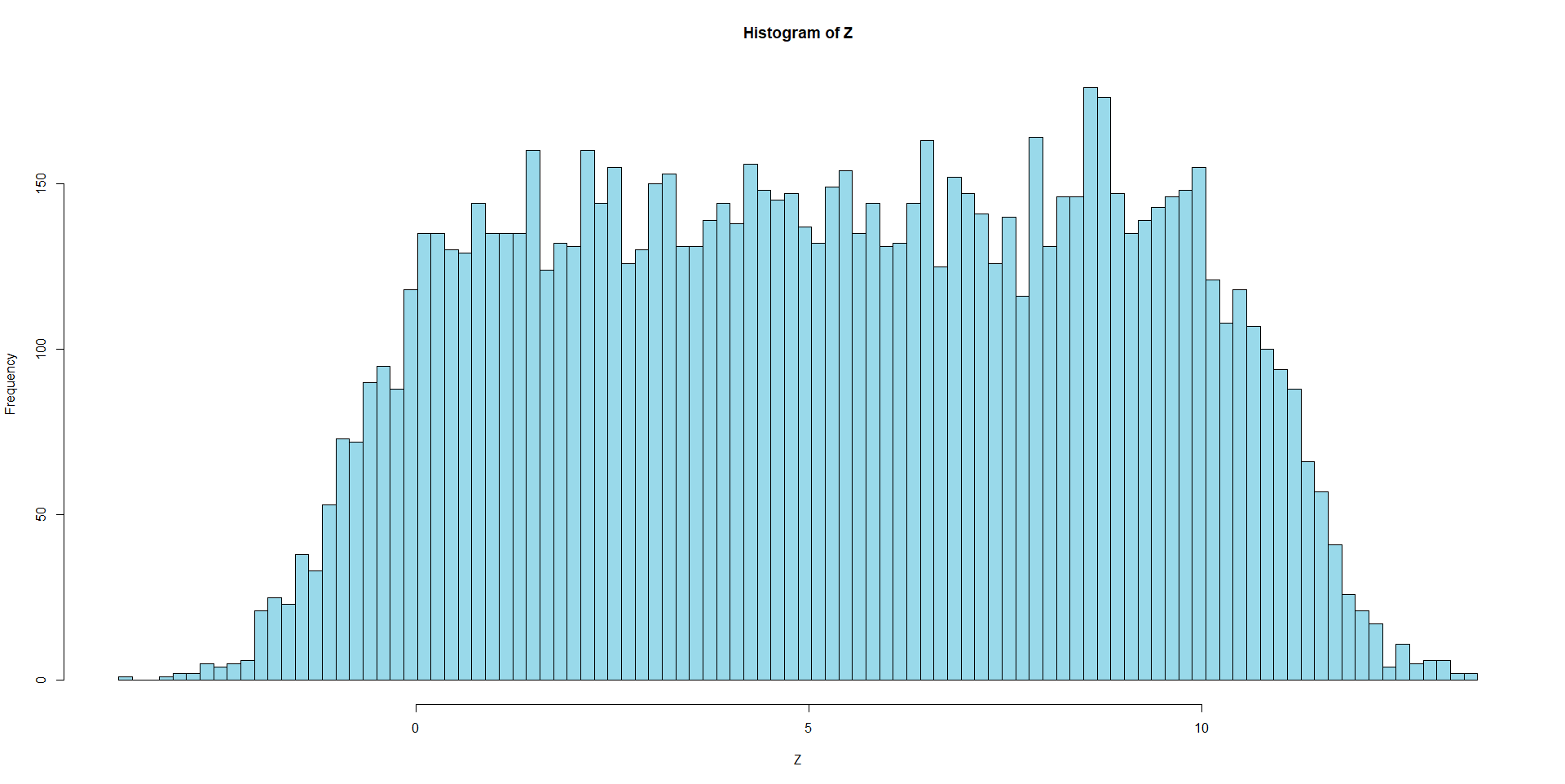} 
\includegraphics[width=140pt]{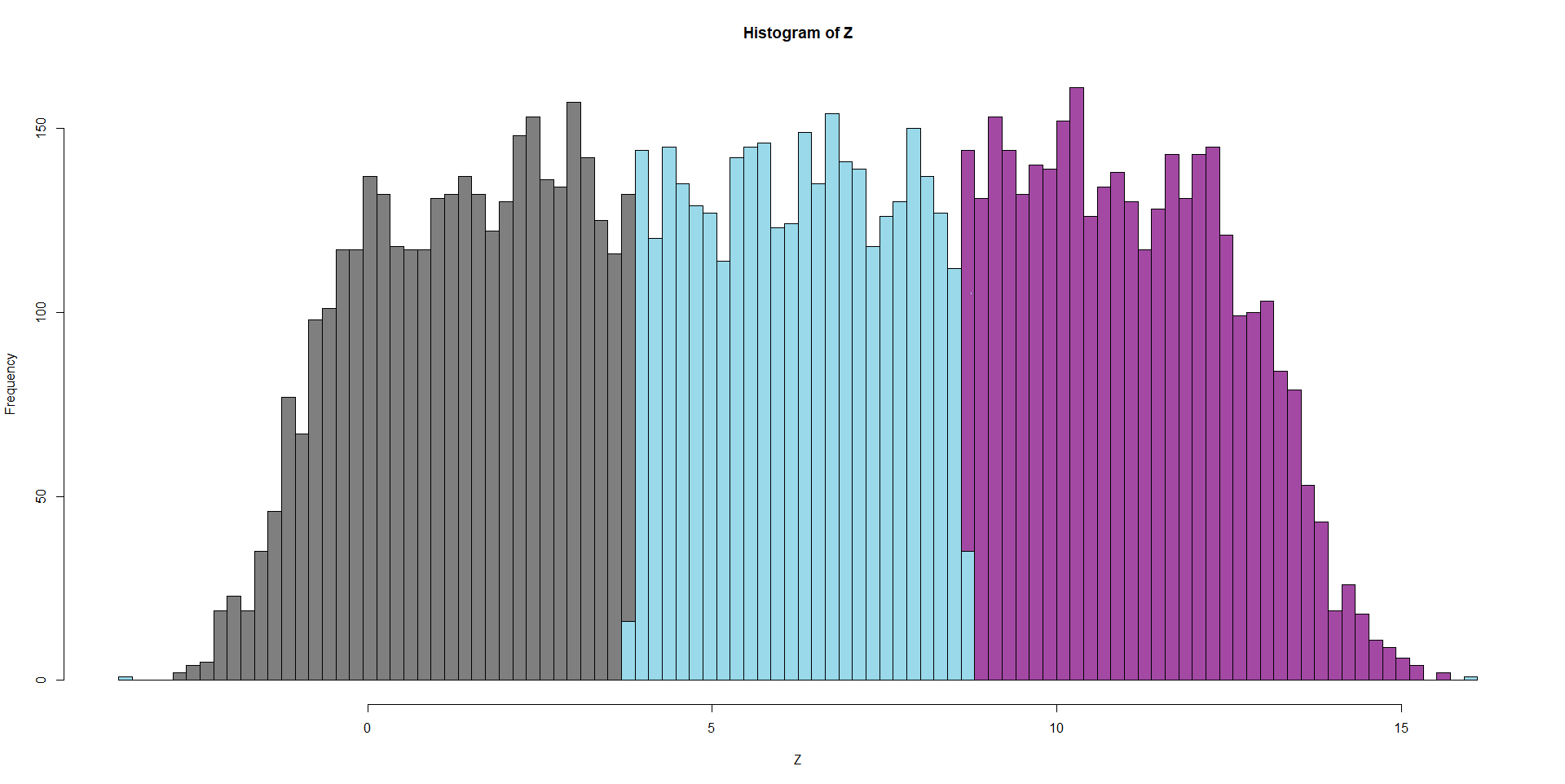} \\
\includegraphics[width=140pt]{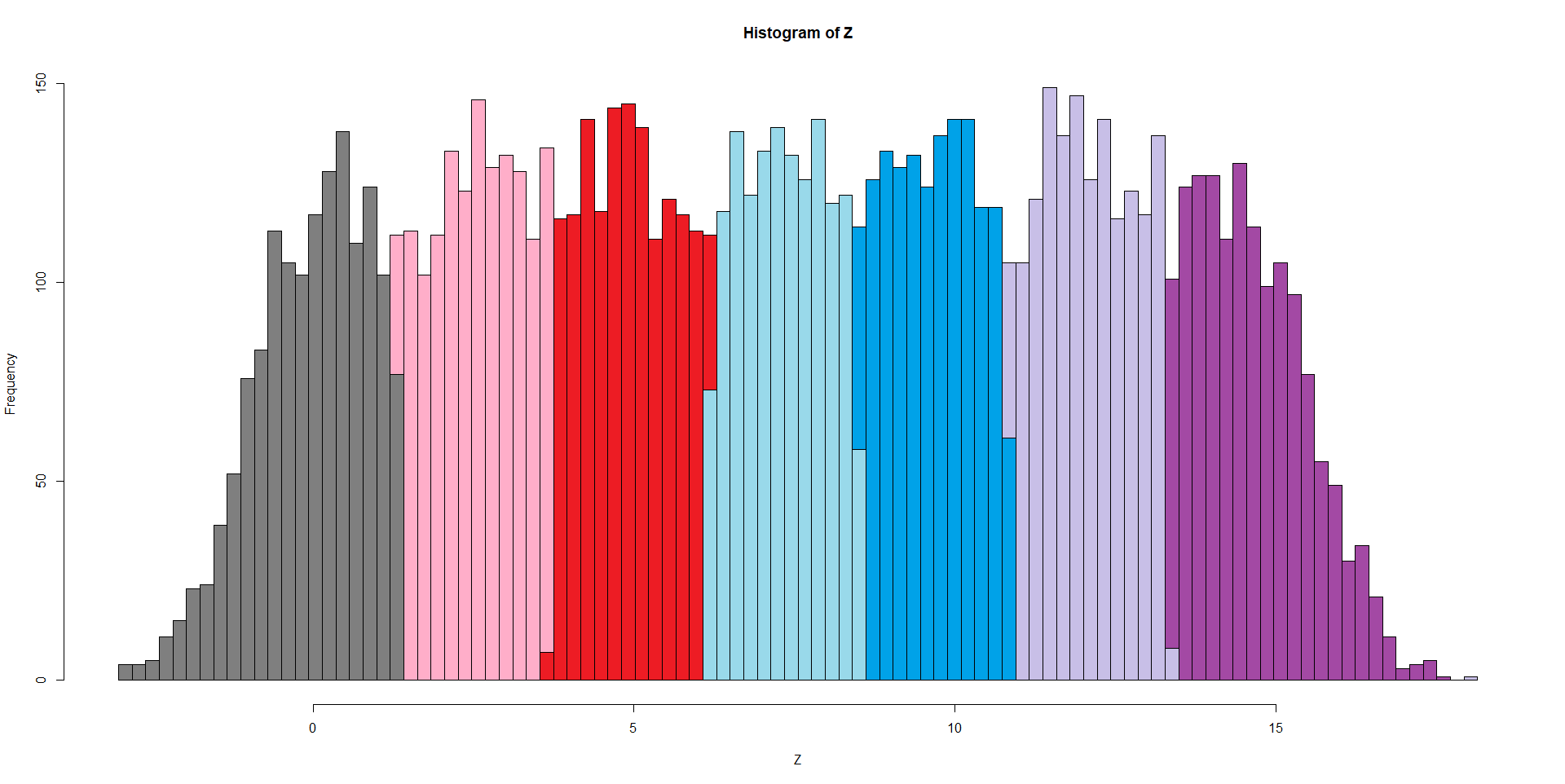}
\includegraphics[width=140pt]{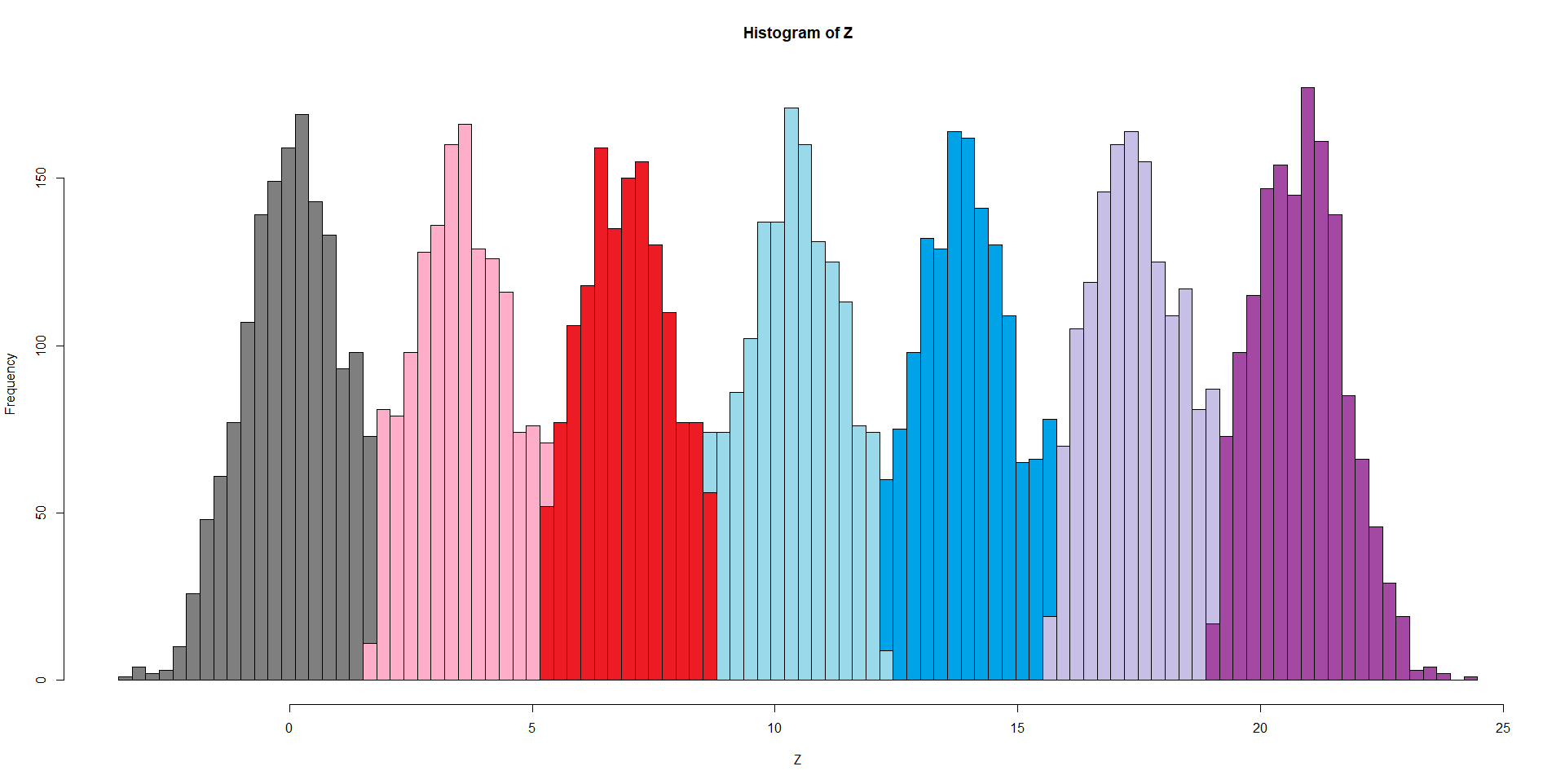} \\
\caption{Histograms of the data regrouped by class using the  exponential family \eqref{entrop:eq:biexp:density} (corresponding to the right column of Table \ref{entrop:table:case:r7}) with, from top to bottom and left to right, $\mu^\star = 0.50*(2 \sqrt{3})$, $\mu^\star =  0.60*(2 \sqrt{3})$,  $\mu^\star =  0.70*(2 \sqrt{3})$ and  $\mu^\star =   2 \sqrt{3} $.}
\label{entrop:fig:case:r7:illustr}
\end{center}
\end{figure}

Once again, we observe on Table \ref{entrop:table:case:r7} and Figure \ref{entrop:fig:case:r7:illustr}  the merging tendency of the classes as the clumps of the distribution $P^\star$  get closer to each other.  

\section{Proof of Theorem \ref{entrop:th:phistar}}
\label{entrop:sec:proofs}
\label{entrop:sec:proof:th:phistar}
We assume that there exists $x_0$ in $\{1,\ldots,r\}$ such that $ P^\star \left(\phi^\star_{x_0} \in ]0,1[ \right)>0$. Using Borel-Cantelli Lemma we can therefore consider an $\alpha$ in $]0,\frac{1}{2}[$ satisfying $P^\star \left(\phi^\star_{x_0} \in [\alpha,1-\alpha] \right)>0$. 
Since $\sum_{x=1}^r\phi^\star_x = 1$, we can also assume the existence of an other $x_1\neq x_0$ such that   
 $ P^\star \left(\phi^\star_{x_0} \in [\alpha,1-\alpha]  \mbox{ and }\phi^\star_{x_1} \in [\alpha,1-\alpha] \right)>0$, even if it means choosing a smaller $\alpha$.  Denote by $A_\alpha$ the set  
$$A_\alpha =A_\alpha(x_0,x_1)= \left\{z\in \mathbb{Z} \ | \ \phi^{\star}_{x_0}(z) \in [\alpha,1-\alpha] \mbox{ and } \phi^{\star}_{x_1}(z) \in [\alpha,1-\alpha] \right\}\;,$$

and let $A$ be any measurable subset of $A_\alpha$ satisfying $P^\star(A)>0$. Let $\theta^\star=\left(\theta^\star_x \right)_{x=1}^r$ be elements of $\Theta^r$ such that 
$$\mathbb{H}_{p^\star}(\phi^\star)= \mathbb{H}_{p^\star}(\phi^\star,\theta^\star) \;.$$

Let $\delta$ be a real number satisfying $0<\delta<\alpha$. Let $\phi$ be such that: For all $x\neq x_0,x_1$, $\phi_{x} = \phi^\star_{x}$, and, for $x\in \{x_0,x_1\}$,
\begin{align*}
\phi_{x_0}(z) = \phi^\star_{x_0}(z) \mbox{ if } z \not\in A \;,&
\phi_{x_0}(z) = \phi^\star_{x_0}(z) + \delta \mbox{ if } z  \in A \;,\\ 
\phi_{x_1}(z) = \phi^\star_{x_1}(z) \mbox{ if } z \not\in A \;,&
\phi_{x_1}(z) = \phi^\star_{x_1}(z) - \delta  \mbox{ if } z\in A \;.\\ 
\end{align*}
By the definition of $A$, $\phi$ belongs necessarily to $\Phi_r$.  We now explicit  $\mathbb{H}_{p^\star}(\phi,\theta^\star)$:
\begin{align*}
\mathbb{H}_{p^\star}(\phi,\theta^\star) = & - \sum_{x=1}^r \int_{\mathbb{Z}} \log \left[ g_{\theta^\star_x}(z) \nu^\phi(x) \right] \phi_x(z) \rm d P^\star(z) \;,\\
=& - \sum_{x\not\in \{x_0,x_1\}} \int_{\mathbb{Z}} \log \left[ g_{\theta^\star_x}(z) \nu^{\phi^\star}(x) \right] \phi^\star_x(z) \rm d P^\star(z) \\
& - \int_{A^c} \log \left[ g_{\theta^\star_{x_0}}(z) \left( \nu^{\phi^\star}(x_0) + \delta P^\star\left(A \right)\right) \right] \phi^\star_{x_0}(z) \rm d P^\star(z)\\
& - \int_{A^c} \log \left[ g_{\theta^\star_{x_1}}(z) \left(\nu^{\phi^\star}(x_1) - \delta P^\star\left(A \right) \right) \right] \phi^\star_{x_1}(z) \rm d P^\star(z)\\
& - \int_{A} \log \left[ g_{\theta^\star_{x_0}}(z)  \left( \nu^{\phi^\star}(x_0) + \delta P^\star\left(A \right)\right)\right] \left(\phi^\star_{x_0}(z) + \delta\right) \rm d P^\star(z)\\
& - \int_{A} \log \left[ g_{\theta^\star_{x_1}}(z) \left(\nu^{\phi^\star}(x_1) - \delta P^\star\left(A \right) \right) \right] \left(\phi^\star_{x_1}(z) - \delta\right) \rm d P^\star(z)\;,
\end{align*}
 Using that, for all $x\neq x_0,x_1$, $ \nu^\phi(x)= \nu^{\phi^\star}(x)$, $ \nu^\phi(x_0)= \nu^{\phi^\star}(x_0) + \delta P^\star\left(A  \right)$ and $ \nu^\phi(x_1)= \nu^{\phi^\star}(x_1) - \delta P^\star\left(A \right)$. Now 
 
 \begin{align*}
\mathbb{H}_{p^\star}(\phi,\theta^\star) -\mathbb{H}_{p^\star}(\phi^\star) =& - \int_{\mathbb{Z}} \log \left(  1 + \frac{\delta P^\star\left(A \right)}{ \nu^{\phi^\star}(x_0) }     \right)  \phi^\star_{x_0}(z) \rm d P^\star(z)\\
& - \int_{\mathbb{Z}} \log  \left(1 - \frac{\delta P^\star\left(A\right)}{ \nu^{\phi^\star}(x_1) } \right)  \phi^\star_{x_1}(z) \rm d P^\star(z)\\
& - \delta\int_{A} \log \left[ g_{\theta^\star_{x_0}}(z)  \left( \nu^{\phi^\star}(x_0) + \delta P^\star\left(A \right)\right)\right]    \rm d P^\star(z)\\
& +\delta \int_{A} \log \left[ g_{\theta^\star_{x_1}}(z) \left(\nu^{\phi^\star}(x_1) - \delta P^\star\left(A\right) \right) \right]   \rm d P^\star(z)\;,
\end{align*}

Using A\ref{entrop:hyp:log:g:integrable} and the dominated convergence theorem, $ \lim_{\delta \to 0} \frac{\mathbb{H}_{p^\star}(\phi,\theta^\star) -\mathbb{H}_{p^\star}(\phi^\star)}{\delta}$ exists and is equal to:
 \begin{align}
 \lim_{\delta \to 0} \frac{\mathbb{H}_{p^\star}(\phi,\theta^\star) -\mathbb{H}_{p^\star}(\phi^\star)}{\delta } =&-  P^\star(A) + P^\star(A) \label{entrop:eq:lim:ratio:phistar}\\
& - \int_{A} \log \left[ g_{\theta^\star_{x_0}}(z)  \left( \nu^{\phi^\star}(x_0) \right)\right] \rm d P^\star(z)\nonumber \\ 
& + \int_{A} \log \left[ g_{\theta^\star_{x_1}}(z) \left(\nu^{\phi^\star}(x_1)  \right) \right]  \rm d P^\star(z)\;.\nonumber
\end{align}

Now, either there exists $\delta>0$ such that $\mathbb{H}_{p^\star}(\phi,\theta^\star) <\mathbb{H}_{p^\star}(\phi^\star)$, proving Theorem \ref{entrop:th:phistar}, or, for all $\delta>0$, $\mathbb{H}_{p^\star}(\phi,\theta^\star) \ge \mathbb{H}_{p^\star}(\phi^\star)$ and the numerator   $\mathbb{H}_{p^\star}(\phi,\theta^\star) -\mathbb{H}_{p^\star}(\phi^\star)$ is non negative while the sign of  the denominator  is the sign of $\delta  $ thus, the limit in \eqref{entrop:eq:lim:ratio:phistar} is necessarily zero and
\begin{equation}
\int_{A} \log \left[ g_{\theta^\star_{x_0}}(z)  \left( \nu^{\phi^\star}(x_0) \right)\right] \mathrm{d} P^\star(z) =  \int_{A} \log \left[ g_{\theta^\star_{x_1}}(z) \left(\nu^{\phi^\star}(x_1)  \right) \right]\mathrm{d} P^\star(z)\;,
\end{equation}
which can be rewritten as:
\begin{equation}
\label{entrop:eq:g1egalg2}
\frac{1}{P^\star(A)}\int_{A } \log \left[\frac{ g_{\theta^\star_{x_0}}(z)}{ g_{\theta^\star_{x_1}}(z)} \right]\mathrm{d} P^\star(z) =   \log\left( \frac{ \nu^{\phi^\star}(x_1)}{ \nu^{\phi^\star}(x_0)} \right) \;.
\end{equation}

\textbf{Case 1 } : If A\ref{entrop:hyp:Z:metric}-\ref{entrop:hyp:theta:coincident} are satisfied:  Then, there exists an open subset $U$ of $A_\alpha$ such that for all $z_0$ in $U$, there exists $\varepsilon_0$ such that, for all $\varepsilon<\varepsilon_0$, $B_d(z_0,\varepsilon) \subset U$, and $P^\star \left( B_d(z_0,\varepsilon)\right)>0$. Applying \eqref{entrop:eq:g1egalg2} to $A=B_d(z_0,\varepsilon) $ gives, for all $\epsilon_0>\epsilon>0$, 

$$\frac{1}{P^\star(B_d(z_0,\varepsilon))}\int_{B_d(z_0,\varepsilon) } \log \left[\frac{ g_{\theta^\star_{x_0}}(z)}{ g_{\theta^\star_{x_1}}(z)} \right]\mathrm{d} P^\star(z) =  \log\left( \frac{ \nu^{\phi^\star}(x_1)}{ \nu^{\phi^\star}(x_0)} \right) \;.$$
In this case, $z \mapsto \log \left[\frac{ g_{\theta^\star_{x_0}}(z)}{ g_{\theta^\star_{x_1}}(z)} \right]$ is  continuous at $z_0$ and 
\begin{align*}
\log \left[\frac{ g_{\theta^\star_{x_0}}(z_0)}{ g_{\theta^\star_{x_1}}(z_0)} \right] &= \lim_{\epsilon\to 0}\frac{1}{P^\star(B_d(z_0,\varepsilon))}\int_{B_d(z_0,\varepsilon) } \log \left[\frac{ g_{\theta^\star_{x_0}}(z)}{ g_{\theta^\star_{x_1}}(z)} \right]\mathrm{d} P^\star(z) \\
&=   \log\left( \frac{ \nu^{\phi^\star}(x_1)}{ \nu^{\phi^\star}(x_0)} \right) \;.
\end{align*} 
Therefore, for all $z$ in $U$
 
\begin{equation*}
g_{\theta^\star_{x_0}}(z)     =  \frac{ \nu^{\phi^\star}(x_0)}{ \nu^{\phi^\star}(x_1)}   g_{\theta^\star_{x_1}}(z)  \;,
\end{equation*} 
and, using A\ref{entrop:hyp:theta:coincident}, necessarily,  $ g_{\theta^\star_{x_0}}= g_{\theta^\star_{x_1}}$ and $ \nu^{\phi^\star}(x_0) =  \nu^{\phi^\star}(x_1)$. 
Now defining, for all $x\neq x_0,x_1$, $\phi_x = \phi^\star_x$, for $x=x_0$, $\phi_{x_0} = \phi^\star_{x_0} + \phi^\star_{x_1}$  and, for $x=x_1$, $\phi_{x_1} = 0$,  then, for all $x\neq x_0,x_1$, $\nu^\phi(x) =  \nu^{\phi^\star}(x)$, $\nu^\phi(x_0) =  \nu^{\phi^\star}(x_0) +\nu^{\phi^\star}(x_1) $ and $\nu^\phi(x_1)=0$
\begin{align}
\mathbb{H}_{p^\star}(\phi,\theta^\star) -\mathbb{H}_{p^\star}(\phi^\star) =  & - \int_{\mathbb{Z}} \log \left[ g_{\theta^\star_{x_0}}(z) \left(  \nu^{\phi^\star}(x_0) +\nu^{\phi^\star}(x_1)  \right)\right] \left(\phi^\star_{x_0}(z) + \phi^\star_{x_1}(z)\right) \rm d P^\star(z) \nonumber \\
&+ \int_{\mathbb{Z}} \log \left[ g_{\theta^\star_{x_0}}(z) \left(  \nu^{\phi^\star}(x_0)  \right)\right] \left(\phi^\star_{x_0}(z) \right) \rm d P^\star(z) \nonumber \\
&+  \int_{\mathbb{Z}} \log \left[ g_{\theta^\star_{x_0}}(z) \left( \nu^{\phi^\star}(x_1)  \right)\right] \left( \phi^\star_{x_1}(z)\right) \rm d P^\star(z) \nonumber \\
 =  & -   \left(  \nu^{\phi^\star}(x_0) +\nu^{\phi^\star}(x_1)  \right)   \log  \left(  \nu^{\phi^\star}(x_0) +\nu^{\phi^\star}(x_1)  \right)  \label{entrop:eq:diffH:case1} \\ 
&+ \left(  \nu^{\phi^\star}(x_0)  \right)    \log  \left(  \nu^{\phi^\star}(x_0)  \right)  \nonumber \\
&+  \left( \nu^{\phi^\star}(x_1)  \right)   \log    \left( \nu^{\phi^\star}(x_1)  \right)\;. \nonumber  
\end{align}
$\nu^{\phi^\star}(x_0)$ and $\nu^{\phi^\star}(x_1)$ being positive, then the right hand side in \eqref{entrop:eq:diffH:case1} is negative and thus
$$\mathbb{H}_{p^\star}(\phi,\theta^\star) -\mathbb{H}_{p^\star}(\phi^\star) <0\;,$$ 
concluding the proof in the  \textbf{case 1}.

\textbf{Case 2 } : If A\ref{entrop:hyp:ident:discret} is satisfied, then $P^\star = \sum_{i=0}^N p^\star(z_i) \delta_{z_i}$ where $1\leq N\leq \infty$, for all $i$ $z_i\in \mathbb{Z}$ and $p^\star(z_i)>0$ with  $\sum_{i=0}^N p^\star(z_i)=1$. If we intend to use the same scheme of proof as in \textbf{Case 1 }, the open balls $B(z_0,\varepsilon)$ are made of the single element $\{z_0\}$ for $\varepsilon$ small enough. The equality, up to a constant, between $g_{\theta^\star_{x_0}}(z)$ and $g_{\theta^\star_{x_0}}(z)$ does not necessarily hold for an infinite amount of $z$'s in this case. However,  arguments adapted to the discrete case paired with Equation \eqref{entrop:eq:g1egalg2} achieve the same result which is  the construction of a $\phi$ with lower entropy than $\phi^\star$. 

Denote by $A_0$ the set of all $z_0$ in $ \{z_i\}_{i=1}^N$ such that there exists $x$ in $\{1,\ldots,r\}$ satisfying $0<\phi^\star_x(z_0) <1$. We assumed at the beginning of the proof  that $A_0$ is non empty. For all $z_0$ in $A_0$ choose arbitrarily one $x_{0}$ such that $0<\phi^\star_{x_0}(z_0) <1$. Therefore $x_0$ depends on $z_0$ in $A_0$ that is considered.  For all $x_1$ satisfying $0<\phi^\star_{x_1}(z_0) <1$,  one can embed $A=\{z_0\}$ in $A_\alpha(x_0,x_1)$. Moreover  $P^\star(A)=p^\star(z_0) >0$, we can therefore apply \eqref{entrop:eq:g1egalg2} to $A=\{z_0\}$ which gives:
$$\frac{1}{p^\star(z_0)} \log\left(  \frac{  g_{   \theta^\star_{x_{0}}   }(z_0)  }{ g_{   \theta^\star_{x_1}  }(z_0)  } \right)p^\star(z_0) = \log \left( \frac{\nu^{\phi^\star}(x_1)}{\nu^{\phi^\star}(x_0)}\right)\;,$$
and then
\begin{equation}
\label{entrop:eq:g1egalg2:discret}
 g_{   \theta^\star_{x_1}  } (z_0)  \nu^{\phi^\star}(x_1) = g_{   \theta^\star_{x_{z_0}}   }(z_0) \nu^{\phi^\star}(x_0)   \;,
\end{equation}

Below, we use the notation $A_0^c$ to designate the (possibly empty) set of all $z$ in $\{z_i\}_{i=1}^N$ that do not belong to $A_0$.  Define $\phi_x(z)$ for all $x$ in $\{1,\ldots,r\}$ and $z$ in $\{z_i\}_{i=1}^N$:
\begin{itemize}
\item[$\cdot$] For all $z\in A_0^c$, $\phi_x(z)=\phi^\star_x(z)$
\item[$\cdot$] For all  $z_0\in A_0$,
\begin{itemize}
\item[$\cdot$] if $x$ is such that $\phi^\star_x(z_0)\in \{0,1\}$, $\phi_x(z_0)=\phi^\star_x(z_0)$,
\item[$\cdot$] if $x$ is such that $\phi^\star_x(z_0)\in ]0,1[$ and if $x\neq x_0$, $\phi_x(z_0)=0$,
\item[$\cdot$] $\phi_{x_0}(z_0)=\sum_{x |\phi^\star_x(z_0)\in ]0,1[ }\phi^\star_x(z_0)$.
\end{itemize}
\end{itemize}

Then using relation \eqref{entrop:eq:g1egalg2:discret} we can show that 
\begin{align*}
 \mathbb{H}_{P^\star} \left( \phi^\star\right) =& - \Bigg[ \sum_{x=1}^r \sum_{z \in \{z_i\}} \log\left( g_{\theta^\star_x}(z) \nu^{\phi^\star}(x)\right) \phi^\star_x(z) p^\star(z) \Bigg]\\ 
=& - \Bigg[ \sum_{x=1}^r \sum_{z \in \{z_i\}} \log\left( g_{\theta^\star_x}(z) \nu^{\phi^\star}(x)\right) \phi_x(z) p^\star(z) \Bigg]\;.\\
\end{align*}

Moreover 
\begin{align*}
 \mathbb{H}_{P^\star} \left( \phi ,\theta^\star\right)  =& - \Bigg[ \sum_{x=1}^r \sum_{z \in \{z_i\}} \log\left( g_{\theta^\star_x}(z) \nu^{\phi}(x)\right) \phi_x(z) p^\star(z) \Bigg]\;,\\
\end{align*}
where $\nu^{\phi}(x) =  \sum_{z \in \{z_i\}}\phi_x(z) p^\star(z) $. Thus 

\begin{align*}
 \mathbb{H}_{P^\star} \left( \phi ,\theta^\star\right) -\mathbb{H}_{P^\star} \left( \phi^\star\right) =& - KL\left(\nu^{\phi}||\nu^{\phi^\star} \right)\;.\\
\end{align*}

$A_0$ being non empty, necessarily $\nu^{\phi}\neq\nu^{\phi^\star}$  and $- KL\left(\nu^{\phi}||\nu^{\phi^\star} \right)<0$, concluding the proof in \textbf{Case 2}.

\bibliographystyle{unsrtnat} 
\bibliography{bibtex.bib}

\begin{thebibliography}{19}
\providecommand{\natexlab}[1]{#1}
\providecommand{\url}[1]{\texttt{#1}}
\expandafter\ifx\csname urlstyle\endcsname\relax
  \providecommand{\doi}[1]{doi: #1}\else
  \providecommand{\doi}{doi: \begingroup \urlstyle{rm}\Url}\fi

\bibitem[Bryant(1991)]{bryant:1991}
Peter Bryant.
\newblock Large-sample results for optimization-based clustering methods.
\newblock \emph{Journal of Classification}, 8\penalty0 (1):\penalty0 31--44,
  1991.
\newblock URL
  \url{https://EconPapers.repec.org/RePEc:spr:jclass:v:8:y:1991:i:1:p:31-44}.

\bibitem[Celeux and Govaert(1992)]{celeux:1992}
Gilles Celeux and G{\'e}rard Govaert.
\newblock A classification em algorithm for clustering and two stochastic
  versions.
\newblock \emph{Computational statistics \& Data analysis}, 14\penalty0
  (3):\penalty0 315--332, 1992.

\bibitem[Biernacki et~al.(2000)Biernacki, Celeux, and Govaert]{biernacki:2000}
Christophe Biernacki, Gilles Celeux, and G\'{e}rard Govaert.
\newblock Assessing a mixture model for clustering with the integrated
  completed likelihood.
\newblock \emph{IEEE Trans. Pattern Anal. Mach. Intell.}, 22\penalty0
  (7):\penalty0 719–725, jul 2000.
\newblock ISSN 0162-8828.
\newblock \doi{10.1109/34.865189}.
\newblock URL \url{https://doi.org/10.1109/34.865189}.

\bibitem[Baudry et~al.(2012)Baudry, Maugis, and Michel]{baudry:2012}
Jean-Patrick Baudry, Cathy Maugis, and Bertrand Michel.
\newblock Slope heuristics: overview and implementation.
\newblock \emph{Statistics and Computing}, 22\penalty0 (2):\penalty0 455--470,
  2012.

\bibitem[Celisse et~al.(2012)Celisse, Daudin, and Pierre]{celisse:2012}
Alain Celisse, Jean-Jacques Daudin, and Laurent Pierre.
\newblock {Consistency of maximum-likelihood and variational estimators in the
  stochastic block model}.
\newblock \emph{Electronic Journal of Statistics}, 6\penalty0 (none):\penalty0
  1847 -- 1899, 2012.
\newblock \doi{10.1214/12-EJS729}.
\newblock URL \url{https://doi.org/10.1214/12-EJS729}.

\bibitem[Quost and Denoeux(2016)]{quost:2016}
Benjamin Quost and Thierry Denoeux.
\newblock Clustering and classification of fuzzy data using the fuzzy em
  algorithm.
\newblock \emph{Fuzzy Sets and Systems}, 286:\penalty0 134--156, 2016.

\bibitem[Spurek et~al.(2017)Spurek, Tabor, and Byrski]{spurek:2017}
P.~Spurek, J.~Tabor, and K.~Byrski.
\newblock Active function cross-entropy clustering.
\newblock \emph{Expert Systems with Applications}, 72:\penalty0 49--66, 2017.
\newblock ISSN 0957-4174.
\newblock \doi{https://doi.org/10.1016/j.eswa.2016.12.011}.

\bibitem[Dempster et~al.(1977)Dempster, Laird, and
  Rubin]{dempster:laird:rubin:1977}
A.~P. Dempster, N.~M. Laird, and D.~B. Rubin.
\newblock Maximum likelihood from incomplete data via the {EM} algorithm.
\newblock \emph{J. Roy. Statist. Soc. B}, 39\penalty0 (1):\penalty0 1--38 (with
  discussion), 1977.

\bibitem[Baum et~al.(1970)Baum, Petrie, Soules, and Weiss]{baum:1970}
Leonard~E Baum, Ted Petrie, George Soules, and Norman Weiss.
\newblock A maximization technique occurring in the statistical analysis of
  probabilistic functions of markov chains.
\newblock \emph{The annals of mathematical statistics}, 41\penalty0
  (1):\penalty0 164--171, 1970.

\bibitem[Akaike(1973)]{akaike:1973}
H~Akaike.
\newblock Information theory and an extension of the maximum likelihood
  principle.
\newblock In \emph{2nd International Symposium on Information Theory}, pages
  267--281. Akad{\'e}miai Kiad{\'o} Location Budapest, Hungary, 1973.

\bibitem[Mallows(1973)]{mallows:1973}
C.~L. Mallows.
\newblock Some comments on cp.
\newblock \emph{Technometrics}, 15\penalty0 (4):\penalty0 661--675, 1973.
\newblock ISSN 00401706.
\newblock URL \url{http://www.jstor.org/stable/1267380}.

\bibitem[Massart(2007)]{massart:2007}
Pascal Massart.
\newblock \emph{Concentration inequalities and model selection}, volume~6.
\newblock Springer, 2007.

\bibitem[Biernacki and Govaert(1997)]{biernacki:1997}
Christophe Biernacki and G\'{e}rard Govaert.
\newblock Using the classification likelihood to choose the number of clusters.
\newblock \emph{Computing Science and Statistics}, 29\penalty0 (2):\penalty0
  451--457, 1997.

\bibitem[Shannon(1948)]{shannon:1948}
Claude~Elwood Shannon.
\newblock A mathematical theory of communication.
\newblock \emph{The Bell system technical journal}, 27\penalty0 (3):\penalty0
  379--423, 1948.

\bibitem[Gassiat(2018)]{gassiat:2018}
{\'E}lisabeth Gassiat.
\newblock \emph{Universal Coding and Order Identification by Model Selection
  Methods}.
\newblock Springer, 2018.
\newblock ISBN 9783319962627.

\bibitem[Dumont(2022)]{dumont:supp:2022}
Thierry Dumont.
\newblock Supplement paper to "adaptive clustering by minimization of the
  mixing entropy criterion".
\newblock \emph{Prepublication}, 2022.

\bibitem[Kullback(1997)]{kullback:1997}
Solomon Kullback.
\newblock \emph{Information theory and statistics}.
\newblock Courier Corporation, 1997.

\bibitem[Rudin(1991)]{rudin:1991}
W.~Rudin.
\newblock \emph{Functional Analysis}.
\newblock International series in pure and applied mathematics. McGraw-Hill,
  1991.
\newblock ISBN 9780070542365.

\bibitem[Capp\'{e} et~al.(2005)Capp\'{e}, Moulines, and
  Ryd\'{e}n]{cappe:moulines:ryden:2005}
O.~Capp\'{e}, E.~Moulines, and T.~Ryd\'{e}n.
\newblock \emph{Inference in Hidden {M}arkov Models}.
\newblock Springer, 2005.

\end{thebibliography}


\begin{thebibliography}{1}
\providecommand{\natexlab}[1]{#1}
\providecommand{\url}[1]{\texttt{#1}}
\expandafter\ifx\csname urlstyle\endcsname\relax
  \providecommand{\doi}[1]{doi: #1}\else
  \providecommand{\doi}{doi: \begingroup \urlstyle{rm}\Url}\fi

\bibitem[Dumont(2022)]{dumont:2022}
Thierry Dumont.
\newblock Adaptive clustering by minimization of the mixing entropy criterion.
\newblock \emph{Prepublication}, 2022.

\end{thebibliography}

\end{document}


\maketitle
\begin{abstract}
In the supplementary material, we state and prove several technical results used in the paper \cite{dumont:2022}.
\end{abstract}

\section{Mixing entropy in the binary case}
\label{entrop:sec:appendix:prop:binary}

We consider here the binary case where $\mathbb{Z}=\{0,1\}$ and where the parameters $\theta = (\mu_0,\mu_1)$  satisfy, for $z$ in $\{0,1\}$, $\mu_z\ge 0$ and $\mu_0+\mu_1 = 1$. If $\delta_z$ denotes the Dirac distribution on $\{z\}$, we define $g_{\theta} =  \mu_0\delta_0 + \mu_1\delta_1$ and assume that  $P^\star$ belongs to the family $\{g_\theta\;,\;\theta \in \Theta\}$: $$P^\star = \mu^\star_0\delta_0 + \mu^\star_1\delta_1 =g_{\theta^\star}\;.$$ 

The objective of this section is to provide a full description of $\mathcal{D}^\star_r$ for any $r\ge 1$ in this case. 

Let $r \ge 1$, $\nu$ be in  $\mathcal{M}_1(\{1,\ldots,r\}) $ and, for all $x$ in $\{1,\ldots,r\}$, $G_x = g_{x}(0) \delta_0 + g_{x}(1)\delta_1$ in $\mathcal{M}_1 \left( \{0,1\}\right)$, with $ g_{x}(0) + g_{x}(1) = 1$. Notice that for all $\theta = \left( \theta_x\right)_{x=1,\ldots,r} $ and all $x$ in $\{1,\ldots,r\}$,

\begin{align*}
\mathbb{H}\left(G_x \ || \ g_{\theta_x} \right) \ge \mathbb{H}\left(G_x \ || \ g_x \right) = H\left(g_x \right)\;,
\end{align*}
thus
\begin{equation*}
\mathbb{H}\left( \nu ,\left( G_x \right)_{x=1,\ldots,r}\right) =  H(\nu) + \sum_{x^1}^r \nu(x) H(g_x) \;.
\end{equation*}
Assume now that $\left( \nu ,\left( G_x \right)_{x=1,\ldots,r}\right)$ belongs to $\left(\mathcal{C}^\star_r \right)$. We assume that $x\mapsto \nu(x)$ is decreasing (even if it means applying a label permutation). Then
$$ \sum_{x=1}^r \nu(x) g_{x}(0) = \mu^\star_0\;\mbox{and}\;  \sum_{x=1}^r \nu(x) g_{x}(1) = \mu^\star_1\;.$$
Therefore,
\begin{multline*}
\mathbb{H}\left( \nu ,\left( G_x \right)_{x=1,\ldots,r}\right) =  - \sum_{x=1}^r \nu(x) \log(\nu(x))  \\- \sum_{x=1}^r \nu(x) \left[g_x(0) \log(g_x(0))+g_x(1) \log(g_x(1))\right]  \;.\\
\end{multline*}

\begin{proposition}
\label{entrop:prop:binary}
For all $r\ge 3$, $\mathcal{D}^\star_r = \mathcal{D}^\star_2$ which means that for all $(\nu , (G_x)_{x=1}^r)$ in $\mathcal{D}^\star_r$, there exits a permutation $\sigma$ of $\{1,\ldots,r\}$ such that, for all $x\ge 3$, $\nu(x)=0$. 

Moreover, for all $((\nu(1),\nu(2)),(G_1,G_2))$ in $\mathcal{D}^\star_2$, either
\begin{itemize}
\item $((\nu(1),\nu(2)),(G_1,G_2))$ belongs to $\mathcal{D}^\star_1$ which means that
\begin{align*}
(\nu(1),\nu(2)) &=(1,0)\mbox{ and }G_1=P^\star \;,\\
\mbox{or }& \\
(\nu(1),\nu(2)) &=(1,0)\mbox{ and }G_2=P^\star \;,
\end{align*}
\item or $\left((\nu(1),\nu(2)),(G_1,G_2)\right)=\left( (\mu^\star_0,\mu^\star_1) , (\delta_0,\delta_1)\right)$
\item or  $\left((\nu(1),\nu(2)),(G_1,G_2)\right)=\left( (\mu^\star_1,\mu^\star_0) , (\delta_1,\delta_0)\right)$.
\end{itemize}

Consequently, if $0<\mu^\star_0,\mu^\star_1<1$, then,
$$\mathcal{D}^\star_1 \varsubsetneq \mathcal{D}^\star_2 = \mathcal{D}^\star_3 = \mathcal{D}^\star_4=\ldots\;,$$
and if  $\mu^\star_0=1$ or $\mu^\star_1=1$ , then, 
$$\mathcal{D}^\star_1 = \mathcal{D}^\star_2 = \mathcal{D}^\star_3 = \mathcal{D}^\star_4=\ldots\;.$$

\end{proposition}

\begin{proof}

Denote by $z^\star$ the value in  $\{0,1\}$ satisfying $\mu^\star_{z^\star} = \max(\mu^\star_0,\mu^\star_1)$. We now use the shortest notations $\mu^\star :=\mu^\star_{z^\star}$ and, for all $x$ in $\{1,\ldots,r\}$, $p_x :=   g_x(z^\star)$.
Then $\sum_{x=1}^r \nu(x)p_x = \mu^\star$ and
\begin{multline*}
\mathbb{H}\left( \nu ,\left( G_x \right)_{x=1,\ldots,r}\right) =  - \sum_{x=1}^r \nu(x) \log(\nu(x))  \\- \sum_{x=1}^r \nu(x) \left[p_x \log(p_x )+ \left(1-p_x \right) \log(1-p_x)\right]  \;.\\
\end{multline*}

We first minimize $\mathbb{H}$ over $(p_x)_{x=1}^r$ when $\nu$ is kept fixed.  Since $x\mapsto x\log(x) + (1-x) \log(1-x)$ is maximized on $[0,1]$ by $x = 0$ or $x=1$, for a given  $\nu$,  $\mathbb{H}$ is minimized along  $(p_x)_{x=1}^r$ satisfying $\sum_{x=1}^r \nu(x)p_x = \mu^\star$  by setting $p_x$ the closest possible to $0$ or $1$.
Define:
$$r_\nu =  \min \left\{x_{0} \ge 1\ | \ \sum_{x=1}^{x_0} \nu(x) \ge \mu^\star\right\}\;,$$  
and, for all $x\in \{1,\ldots,r\}$, 
\begin{align*}
&p^\nu_x = 1 \mbox{ if }x<r_\nu\;, \\
&p^\nu_{r_\nu} = \frac{\mu^\star - \sum_{x=1}^{r_\nu -1} \nu(x) }{\nu(r_\nu)}\;,(\mbox{setting } \sum_{x=1}^{0} \nu(x) :=0)\\
&p^\nu_x = 0 \mbox{ if }x>r_\nu\;,
\end{align*} 

\begin{lemma}
\label{entrop:lemma:binary:optim:px}
$p^\nu_{r_\nu}\in ]0,1]$ and if for all $x$ in $\{1,\ldots,r\}$, $G^\nu_x$ is defined by $G^\nu_x(z^\star) = p^\nu_x $
\begin{multline*}
\min_{  ( G_x)}    \mathbb{H}\left( \nu ,\left( G_x \right)_{x=1,\ldots,r}\right) =\mathbb{H}\left( \nu ,\left( G^\nu_x \right)_{x=1,\ldots,r}\right)\\ 
=  - \sum_{x=1}^r \nu(x) \log(\nu(x))\\  -   \nu(r_\nu ) \left[p^\nu_{r_\nu}  \log(p^\nu_{r_\nu}  )+ \left(1-p^\nu_{r_\nu}  \right) \log(1-p^\nu_{r_\nu} )\right]  \;,
\end{multline*}
\end{lemma}
Where we shortly denote $\min_{  ( G_x)} $ to designate, for a given $\nu$, the minimum over all $( G_x)_{x=1}^r$ satisfying $(\nu, (G_x))$  in $\left(\mathcal{C}^\star_r \right)$. The proof of Lemma \ref{entrop:lemma:binary:optim:px} is postponed in  Section \ref{entrop:sec:appendix:lemma:optim:px}. Now that we can minimize $\mathbb{H}$ for any given $\nu$, we can now tackle the minimization  with respect with $\nu$.
\begin{lemma}
\label{entrop:lemma:Dr}
 Let $r\ge 3$.
\begin{itemize}
\item[$\cdot$]  If $r_\nu >1$.  let $\widetilde{\nu}$ be defined as : 
 \begin{align*}
 \widetilde{\nu}(1) &= \sum_{x<r_\nu } \nu(x) \\
 \widetilde{\nu}(2) &=\nu(r_\nu)\\ 
 \widetilde{\nu}(3) &= \sum_{x>r_\nu } \nu(x) \mbox{ if }r_\nu <r \mbox{ or }\widetilde{\nu}(3)= 0 \mbox{ if }r_\nu =r\;.\\ 
 \end{align*}
(By construction $r_{\widetilde{\nu}} =2$).
 \item[$\cdot$]  If $r_\nu =1$.  let $\widetilde{\nu}$ be defined as : 
 \begin{align*} 
 \widetilde{\nu}(1) &=\nu(r_\nu)\\ 
 \widetilde{\nu}(2) &= \sum_{x>1 } \nu(x)\;.\\ 
 \end{align*}
(By construction $r_{\widetilde{\nu}} =1$)
\end{itemize} 
 In both cases,
 $$\min_{  ( G_x) }    \mathbb{H}\left( \widetilde{\nu} ,\left( G_x \right)_{x=1,\ldots,r}\right) - \min_{  ( G_x) }    \mathbb{H}\left( \nu ,\left( G_x \right)_{x=1,\ldots,r}\right) = H(\widetilde{\nu}) -  H(\nu)\le 0\;,$$
 and the inequality is strict if one of the gathering is strict (i.e. if there exists $x=1,2,3$ such that $\widetilde{\nu}(x) >\nu(x)$). Consequently,  $\mathcal{D}^\star_r = \mathcal{D}^\star_3$ 
 \end{lemma}
 \begin{proof}
 $\widetilde{\nu} ( r_{\widetilde{\nu}} )=\nu(r_\nu ) $ and
 $p^{\widetilde{\nu}}_{r_{\widetilde{\nu}}} = p^\nu_{r_\nu}$. Define $q_1 :=\sum_{x<r_\nu } \nu(x)$ and  $q_2 :=\sum_{x>r_\nu } \nu(x)$ (in the case where $r_\nu <r$), then
\begin{align*}
\min_{  ( G_x) }    \mathbb{H}\left( \widetilde{\nu} ,\left( G_x \right)_{x=1,\ldots,r}\right)  &- \min_{  ( G_x) }    \mathbb{H}\left( \nu ,\left( G_x \right)_{x=1,\ldots,r}\right)\\
&= H(\widetilde{\nu}) -  H(\nu)\\
&= \sum_{x< r_\nu } \nu(x) \log(\nu(x)) -q_1\log(q_1) \\
&\quad\quad\quad+  \sum_{x>r_\nu } \nu(x) \log(\nu(x)) -q_2\log(q_2)\\
&= \sum_{x< r_\nu } \nu(x) \log(\nu(x)/q_1) + \sum_{x>r_\nu } \nu(x) \log(\nu(x)/q_2)  \\
&\le 0 
\end{align*}
 since for all $x< r_\nu$, $\nu(x)/q_1\le 1$  and for all  $x> r_\nu$, $\nu(x)/q_2\le 1$
 \end{proof}

\begin{lemma}
\label{entrop:lemma:D3}
If $r=3$ and $r_\nu = 2$. Define 
\begin{align*}
 \widetilde{\nu}(1) &= \nu(1) + \nu(2) \;,\\
 \widetilde{\nu}(2) &=\nu(3)   \;.
\end{align*}
then $r_{\widetilde{\nu}} = 1$, and, shortly denoting $p:=p^\nu_{r^\nu} $,  
\begin{multline}
\label{entrop:eq:lemma:D3}
\min_{  ( G_x) }    \mathbb{H}\left( \widetilde{\nu} ,\left( G_x \right)_{x=1,\ldots,r}\right) - \min_{  ( G_x) }    \mathbb{H}\left( \nu ,\left( G_x \right)_{x=1,\ldots,r}\right) = \\ 
\nu(1)\log(\nu(1)) + p\nu(2)\log(p\nu(2)) \\
 - \left( \nu(1)+ p\nu(2) \right)\log \left( \nu(1)+ p\nu(2) \right)
\end{multline} 
and 
\begin{equation}
\label{entrop:eq:lemma:D3:strict}
\min_{  ( G_x) }    \mathbb{H}\left( \widetilde{\nu} ,\left( G_x \right)_{x=1,\ldots,r}\right) - \min_{  ( G_x) }    \mathbb{H}\left( \nu ,\left( G_x \right)_{x=1,\ldots,r}\right) < 0  \;.
\end{equation}
 Consequently,  $\mathcal{D}^\star_3 = \mathcal{D}^\star_2$ and all $(\nu,(G_x))$ in $\mathcal{D}^\star_2$ satisfy $r_\nu = 1$.
\end{lemma} 
\begin{proof}
Equation \eqref{entrop:eq:lemma:D3} is a direct consequence of the definition of $p^\nu_{r_\nu}$:
\begin{align*}
 p^\nu_{r_\nu} 								= \frac{\mu^\star - \nu(1)}{\nu(2)}  \;,\;
 p^{\widetilde{\nu}}_{r_{\widetilde{\nu}}} 	= \frac{\mu^\star }{\nu(1)}  
\end{align*} 
\begin{align*}
\min_{  ( G_x) }    \mathbb{H}\left( \widetilde{\nu} ,\left( G_x \right)_{x=1,\ldots,r}\right) =&- \Bigg[\nu(3) \log(\nu(3)) \\
&+ \mu^\star \log(\mu^\star) +(\nu(1)+\nu(2)-\mu^\star)\log(\nu(1)+\nu(2)-\mu^\star\mu) \Bigg]\;,
\end{align*}
and
\begin{align*}
\min_{  ( G_x) }    \mathbb{H}\left( \nu,\left( G_x \right)_{x=1,\ldots,r}\right)=&- \Bigg[\nu(1) \log(\nu(1)) +\nu(3) \log(\nu(3)) \\
&+ (\mu^\star -\nu(1))\log(\mu^\star -\nu(1)) \\
&+(\nu(1)+\nu(2)-\mu^\star)\log(\nu(1)+\nu(2)-\mu^\star) \Bigg]\;.
\end{align*}
Then,
\begin{align*}
\min_{  ( G_x) }    \mathbb{H}\left( \widetilde{\nu} ,\left( G_x \right)_{x=1,\ldots,r}\right) - \min_{  ( G_x) }    \mathbb{H}\left( \nu ,\left( G_x \right)_{x=1,\ldots,r}\right) = & -\Bigg[  \mu^\star \log(\mu^\star) - \nu(1) \log(\nu(1)) \\
&-(\mu^\star -\nu(1))\log(\mu^\star -\nu(1))  \Bigg]\;.
\end{align*}
The latter concludes the proof of \eqref{entrop:eq:lemma:D3} since $\mu^\star =  p \nu(1) + \nu(2)$ and $\mu^\star -\nu(1) = p\nu(2)$. \eqref{entrop:eq:lemma:D3:strict} results from $\nu(1)>0$ (necessarily since $\nu$ is decreasing), $p>0$ and $\nu(2)>0$ (necessarily since $r_\nu = 2$). 

$\mathcal{D}^\star_3 = \mathcal{D}^\star_2$  is a consequence of \eqref{entrop:eq:lemma:D3} and of the case $r_\nu =1$ of Lemma \ref{entrop:lemma:Dr}.
\end{proof}

Lemmas \ref{entrop:lemma:binary:optim:px}, \ref{entrop:lemma:Dr} and \ref{entrop:lemma:D3} allow us to asses that for all $r\ge 2$, $\mathcal{D}^\star_r = \mathcal{D}^\star_2$ and for all $(\nu,(G_x))$ in $\mathcal{D}^\star_2$, $\nu$ satisfy $r_\nu = 1$ which means that   $\nu(1)\ge \mu^\star$. Now, let $\nu = (\nu(1),\nu(2))$ such that  $\nu(1) \ge\mu^\star$, then

\begin{align*}
\min_{  ( G_x) }    \mathbb{H}\left(\nu ,\left( G_x \right)_{x=1,\ldots,r}\right)   =
&  -\Bigg[(1-\nu(1))\log(1-\nu(1)) + (\nu(1) - \mu^\star)\log(\nu(1) - \mu^\star)\Bigg]\\
&  -\mu^\star\log(\mu^\star)
\end{align*}
$\min_{  ( G_x) }    \mathbb{H}\left(\nu ,\left( G_x \right)_{x=1,\ldots,r}\right) $ is minimum for $\nu(1) = 1$ or $\nu(1) =\mu^\star$ and
\begin{align*}
\min_{   (\mathcal{C}^\star_2)}    \mathbb{H}  &=   H\left( P^\star\right)  \;.
\end{align*}

\subsection{Proof of Lemma \ref{entrop:lemma:binary:optim:px}}
\label{entrop:sec:appendix:lemma:optim:px}
The proof consists in finding the minimum of the function:

$$h((p_x)_{x=1}^r) :=  - \sum_{x=1}^r \nu(x) \left(p_x \log(p_x) +(1-p_x)\log(1-p_x) \right)$$
Over all the $(p_x)_{x=1}^r$ on $[0,1]^r$ satisfying   $\sum_{x=1}^r \nu(x) p_x = \mu^\star$. Simply denote $\mu:=\mu^\star$.

We assume that $r=2$ the general case will directly follow by induction.  Recall that we assumed that $x\mapsto \nu(x)$ is decreasing and that $\mu \ge\frac{1}{2}\ge 1-\mu$. First notice that 
\begin{align*}
&\sum_{x=1}^r \nu(x) p_x  =\mu \\
\mbox{iff }& \nu(1)p_1 + \nu(2)p_2  =\mu \\
\mbox{iff }& \nu(1)p_1 + (1- \nu(1))p_2  =\mu \\
\mbox{iff }&  p_1 \in \left[\frac{\nu(1)+\mu-1 }{\nu(1)},\frac{\mu }{\nu(1)}\right] \mbox{ and } p_2 = \frac{\mu-\nu(1)p_1}{1-\nu(1)}\;. \\
\end{align*}  

We now minimize 
$$h(p_1) :=h(p_1,\frac{\mu-\nu(1)p_1}{1-\nu(1)}) $$
with  $p_1$ in $\left[ \max(0,\frac{\nu(1)+\mu-1 }{\nu(1)}) , \min(1,\frac{\mu }{\nu(1)}) \right]$

\textit{Case $\frac{\mu }{\nu(1)}<1 $ :}

There $ \min(1,\frac{\mu }{\nu(1)}) = \frac{\mu }{\nu(1)}$ and $\max(0,\frac{\nu(1)+\mu-1 }{\nu(1)}) = \frac{\nu(1)+\mu-1 }{\nu(1)}$ and  studying  $h(p_1)$ shows that $h(p_1)$ reaches its minimum on one of the two edges. Yet
\begin{align*}
h\left(\frac{\mu }{\nu(1)}\right) &= - \left[\nu(1) \log \left(\frac{\nu(1) - \mu}{\nu(1)} \right) - \mu \log \left( \frac{\nu(1) - \mu}{\mu} \right) \right]\\
h\left(\frac{\nu(1)+\mu-1 }{\nu(1)}\right) &= - \left[\nu(1) \log \left(\frac{\nu(1) - (1-\mu)}{\nu(1)} \right) - (1-\mu) \log \left( \frac{\nu(1) - (1-\mu)}{1-\mu} \right) \right]
\end{align*}
Defining for all $y$ in $[0,\nu(1)]$, $f(y) = - \left[\nu(1) \log \left(\frac{\nu(1) - y}{\nu(1)} \right) - y \log \left( \frac{\nu(1) - y}{y} \right) \right]$, then $f$ satisfies
 \begin{align*}
h\left(\frac{\mu }{\nu(1)}\right) &= f(\mu) \\
h\left(\frac{\nu(1)+\mu-1 }{\nu(1)}\right) &= f(1-\mu) 
\end{align*}
An basic analysis of $f$ shows the symmetry  $f(\nu(1) - y ) = f(y)$ and that $f$ is increasing on $\left[0,\frac{\nu(1)}{2} \right]$. With the assumptions on $\nu(1)$ and $\mu$,  
$ \nu(1)- \mu <1-\mu < \frac{\nu(1)}{2}$. Then $f(\mu) = f( \nu(1)- \mu)<f(1-\mu )$.  Finally  $h\left(\frac{\mu }{\nu(1)}\right)<h\left(\frac{\nu(1)+\mu-1 }{\nu(1)}\right)$, and $h(p_1)$ reaches its minimum on $p_1 =\frac{\mu }{\nu(1)}$ (and $p_2 = 0$). 

\textit{Case $\frac{\mu }{\nu(1)}>1 $ :}

There $ \min(1,\frac{\mu }{\nu(1)}) = 1$ and $\max(0,\frac{\nu(1)+\mu-1 }{\nu(1)}) = \frac{\nu(1)+\mu-1 }{\nu(1)}$ and  studying  $h(p_1)$ shows that $h(p_1)$ reaches its minimum on one of the two edges. Yet
\begin{align*}
h\left(1\right) &= - \Bigg[(1-\nu(1)) \log \left(\frac{(1-\nu(1)) - (1-\mu)}{(1-\nu(1))} \right) \\
&\quad \quad \quad- (1-\mu) \log \left( \frac{(1-\nu(1)) -  (1-\mu)}{ 1-\mu} \right) \Bigg]\\
h\left(\frac{\nu(1)+\mu-1 }{\nu(1)}\right) &= - \left[\nu(1) \log \left(\frac{\nu(1) - (1-\mu)}{\nu(1)} \right) - (1-\mu) \log \left( \frac{\nu(1) - (1-\mu)}{1-\mu} \right) \right]
\end{align*}

Defining for all $y$ in $[1-\mu,1]$, $f(y) = - \left[y \log \left(\frac{y -(1-\mu) }{y} \right) - (1-\mu) \log \left( \frac{y - (1-\mu)}{1-\mu} \right) \right]$, then $f$ satisfies
 \begin{align*}
h\left(1\right) &= f(1-\nu(1)) \\
h\left(\frac{\nu(1)+\mu-1 }{\nu(1)}\right) &= f(\nu(1)) 
\end{align*}
This time a basic study of $f$  shows that $f$ is increasing. Yet the assumptions on  $\nu(1)$ give $\nu(1)\ge 1-\nu(1)$ and $h(p_1)$ reaches its minimum on $p_1 =1$ (and $p_2 = \frac{\mu - \nu(1)}{\nu(2)}$). 
\end{proof}

\bibliographystyle{unsrtnat} 
\bibliography{bibtex.bib}